\documentclass[12pt]{amsart}

\usepackage{amsmath,bm}
\usepackage[color=yellow,textwidth=2cm,colorinlistoftodos]{todonotes}
\usepackage{hyperref}
\usepackage[nameinlink,capitalize]{cleveref}
\usepackage{dsfont}
\usepackage{url}
\usepackage[english]{babel}

\newcommand{\summe}[2]{\displaystyle\sum_{#1}^{#2}}
\newcommand{\summezwei}[2]{\sum_{#1}^{#2}}

\newcommand{\integral}[4]{\displaystyle\int\limits_{#1}^{#2}{#3}\,{#4}}

\newcommand{\indikator}[2]{\mathds{1}_{{#1}}({#2})}
\newcommand{\indikatorzwei}[1]{\mathds{1}_{{#1}}}

\newcommand{\erwartung}[1]{\mathbb{E}({#1})}

\newcommand{\limes}[2]{\lim \limits_{#1 \rightarrow #2}}

\newcommand{\limesinf}[2]{\liminf \limits_{#1 \rightarrow #2} \,}

\newcommand{\ft}[1]{\widehat{{#1}}}
\newcommand{\B}{\mathcal{B}}
\newcommand{\im}{\mathrm{i}}
\newcommand{\C}{\mathbb{C}}

\newcommand{\R}{\mathbb{R}}

\newcommand{\N}{\mathbb{N}}

\newcommand{\Pro}{\mathbb{P} } 
\newcommand{\eps}{\varepsilon} 
\newcommand{\secret}[1]{}

\newcommand{\norm}[1]{\| {#1}\|}
\newcommand{\Li}[1]{\text{L}(\R^{#1})}

\newcommand{\skp}[2]{ \langle {#1},{#2} \rangle}

\newtheorem{theorem}{Theorem}[section]
\newtheorem{lemma}[theorem]{Lemma}
\newtheorem{prop}[theorem]{Proposition}
\newtheorem{cor}[theorem]{Corollary}

\theoremstyle{definition}
\newtheorem{defi}[theorem]{Definition}
\newtheorem{example}[theorem]{Example}

\theoremstyle{remark}
\newtheorem{remark}[theorem]{Remark}

\numberwithin{equation}{section}
\begin{document}
\sloppy
\title[]{Multivariate tempered stable random fields}

\author{D. Kremer}
\address{Dustin Kremer, Department Mathematik, Universit\"at Siegen, 57068 Siegen, Germany}
\email{kremer\@@{}mathematik.uni-siegen.de}

\author{H.-P. Scheffler}
\address{Hans-Peter Scheffler, Department Mathematik, Universit\"at Siegen, 57068 Siegen, Germany}
\email{scheffler\@@{}mathematik.uni-siegen.de}

\date{\today}

%Abstract
\begin{abstract}
Multivariate tempered stable random measures (ISRMs) are constructed and their corresponding space of integrable functions is characterized in terms of a quasi-norm utilizing the so-called Rosinski measure of a tempered stable law. In the special case of exponential tempered ISRMs operator-fractional tempered stable random fields are presented by a moving-average and a harmonizable representation, respectively. 
\end{abstract} 
\maketitle
\begin{center}
Dedicated to the late Mark M. Meerschaert. 
\end{center}
%% Section 1: Introduction
\section{Introduction}
Multivariate tempered stable distributions are obtained by modifying the L\'{e}vy measure of a multivariate stable law by a certain \textit{tempering function} $q(r,\theta)$, see \eqref{eq:17052010} below. This results in tempering the large jumps of the corresponding L\'{e}vy process. See \cite{tappe,rosinski} for the basic theory of tempered stable laws. A prominent subclass of tempered stable laws is obtained by using exponential tempering $q(r,\theta)=q_{\lambda}(r)=e^{- \lambda r}$ for some $\lambda>0$. These exponential tempered stable laws have what is called semi-heavy tails, meaning that the tails decay like a power law for in an intermediate regim and then decay exponentially fast near infinity. They have various applications, see e.g. \cite{mark1,mark5,mark4,mark2,mark3,mark6}. \\
Symmetric $\alpha$-stable random measures $\mathbb{M}(dx)$ are used to construct a large class of stable processes in terms of stochastic integrals via
\begin{equation*}
X(t)=\int f_t(x) \, \mathbb{M}(dx)
\end{equation*}
for appropriate \textit{kernel functions} $f_t$. Prominent examples are fractional stable processes given by either a moving-average or a harmonizable representation, see \cite{SaTaq94} and the literature cited there. By modifying the kernel function of fractional stable processes, so-called tempered fractional stable motions are obtained, see \cite{didier}. The purpose of this paper is to construct fractional tempered stable fields and processes and to study their basic properties. In doing so, we are not modifying the kernel function, but utilizing so-called \textit{tempered stable random measures}. That is, the random noise $\mathbb{M}(dx)$ gets tempered, whereas the kernel functions remains unchanged. \\
Using the recently developed general theory of multivariate \textit{independently scattered random measures} (abbreviated by ISRM) and their corresponding stochastic integrals (see \cite{integral}), in \cref{chapter2} we introduce multivariate tempered stable ISRMs and characterize theirs space of integrable functions utilizing a quasi-norm based the so-called \textit{Rosinski measure} of the underlying tempered stable law. The most prominent example is given by the exponentially tempered ISRM $\mathbb{M}_{\lambda}$, see \cref{17052077} below. Using those we then give examples of operator-fractional tempered stable fields, using ideas and methods of \cite{multi,paper2}. It turns out that the range of possible parameters of those fields is larger than their corresponding stable fields. In \cref{chapter4}, we analyze what happens if the exponential tempering parameter $\lambda$ tends to zero. It turns out that in this case $\mathbb{M}_{\lambda}$ (and the constructed random fields) converge in distribution to theirs stable counterpart $\mathbb{M}_0$. Moreover, it is shown that operator-fractional tempered stable random fields given by the moving-average representation behave locally like operator-fractional stable random fields, using the notion of localisability. 

%% Section 2: 
\section{Preliminary Results} \label{chapter2}
We start with some notation. Write $\min(a,b)=a \wedge b$ and $\max(a,b)=a \vee b$, respectively. Let $\norm{\cdot}$ be the Euclidian norm on $\R^d$ with inner product $\skp{\cdot}{\cdot}$. Denote the set of all linear operators on $\R^d$ by $L(\R^d)$ with $I=I_d$ being the identity operator on $\R^d$, represented as $d \times d$ matrices in each case. Then, by a little abuse of notation, we also write $\norm{\cdot}$ for the operator norm on $L(\R^d)$ (which is induced by the vector space norm $\norm{\cdot}$). Denote the transposed vector of $x$ by $x^t$ and the adjoint matrix of $D$ by $D^{*}$, respectively. Finally, all occurring random vectors are defined on an underlying probability space $(\Omega, \mathcal{A},\Pro)$ and we equip $\R^d$ with its Borel $\sigma$-Algebra $\B(\R^d)$ (accordingly for $L(\R^d))$.
\subsection{Tempered stable distributions and random measures}
Consider $\mu$ to be a tempered $\alpha$-stable distribution (abbreviated by $T \alpha S$) on $\R^d$ in the sense of Definition 2.1 in \cite{rosinski}, where the \textit{stability index} $0<\alpha<2$ is fixed throughout the paper. That means that $\mu$ is infinitely-divisible without Gaussian part and that its \textit{L\'{e}vy measure} $\phi$ in polar coordinates is of the form 
\begin{equation} \label{eq:17052010}
\phi(dr,d \theta )= r^{- \alpha-1} q(r,\theta) \,dr \, \sigma(d\theta).
\end{equation}
Here $\sigma$ is a finite measure on $S^{d-1}:=\{x \in \R^d: \norm{x}=1\}$ and $q:(0, \infty) \times S^{d-1} \rightarrow (0,\infty)$ is a Borel function such that, for every $\theta \in S^{d-1}$, we have:
\begin{itemize}
\item[(i)] $q(\cdot, \theta)$ is completely monotone, i.e.  $(-1)^n \frac{\partial^n}{\partial r^n} q(r,\theta)>0$ for all $r>0$ and $n \in \N$. 
\item[(ii)] $\lim_{r \rightarrow \infty} q(r,\theta)=0$.
\end{itemize}
We call $q$ the tempering function of the $T \alpha S$ distribution $\mu$. It allows to define a measure $\mathcal{R}$ on $\R^d$, which is often referred to as the Rosinski measure. The details can be found in \cite{rosinski}.  Actually, it is uniquely determined by the relation
\begin{equation} \label{eq:14052002}
\phi(A)= \int_{\R^d} \int_{0}^{\infty}  \indikator{A}{rx} r^{- \alpha-1} e^{-r} \,dr \, \mathcal{R}(dx), \quad A \in \B(\R^d).
\end{equation}
It fulfills 
\begin{equation*} 
\mathcal{R}(\{0\})=0, \qquad \int_{\R^d}(\norm{x}^2 \wedge \norm{x}^{\alpha}) \,\mathcal{R}(dx) <\infty,
\end{equation*}
which implies that $\int (1 \wedge \norm{x}^2) \, \mathcal{R}(dx)<\infty$. So $\mathcal{R}$ is a L\'{e}vy measure, too. Also recall the \textit{L\'{e}vy-Khintchine-Formula} (see Theorem 3.1.11 in \cite{thebook}) to verify that $\mu \sim [a, 0, \phi]$, where the vector $a\in \R^d$ is a \textit{shift}. In other words, the characteristic function of $\mu$ is given by $\ft{\mu}(u)=\exp(\psi(u))$, where
\begin{equation} \label{eq:14052001}
\psi(u)= \im \skp{a}{u}+ \int_{\R^d} \left(e^{\im \skp{u}{x} } -1-\frac{\im \skp{u}{x} }{1+\norm{x}^2} \right) \, \phi(dx) , \quad u \in \R^d
\end{equation}
is the corresponding \textit{log-characteristic function} (abbreviated by LCF).
\begin{remark} \label{15052055}
\begin{itemize}
\item[(i)] If $q$ additionally fulfills $\lim_{r \downarrow 0} q(r,\theta)=1$ for every $\theta \in S^{d-1}$, then $\mu$ is called a \textit{proper} $T \alpha S$ distribution in the terminology of \cite{rosinski}. Theorem 2.3 in \cite{rosinski} states that $\mu$ is proper if and only if $\int \norm{x}^{\alpha} \, \mathcal{R}(dx)< \infty$.
\item[(ii)] It is usual to assume that $\mu$ is \textit{full}, i.e. not concentrated on any hyperplane of $\R^d$. Using Proposition 3.1.20 in \cite{thebook} together with \eqref{eq:14052002} it follows that $\mu$ is full if and only if $\mathcal{R}$ is not concentrated on any $d-1$ dimensional subspace of $\R^d$. In this case, following the notation in \cite{thebook}, we say that $\mathcal{R}$ is \textit{$\mathcal{M}$-full}.
\item[(iii)] The function $\psi$ can sometimes be represented more explicitly, cf. \cite{tappe} for the case $d=1$. Moreover, Theorem 2.9 in \cite{rosinski} presents integral representations for $\psi$ in terms of the measure $\mathcal{R}$. However, they are based on a different function within the integral in \eqref{eq:14052001}, which leads to a different shift vector. See Remark 8.4 in \cite{sato} for details. 
\item[(iv)] The previous distinction is negligible if $\mu$ is symmetric. Because in this case $\phi$ is also symmetric and $a=0$, which allows to rewrite \eqref{eq:14052001} as
\begin{equation} \label{eq:14052005}
\psi(u)=\int_{\R^d} (\cos \skp{u}{x}-1) \,\phi(dx) , \quad u \in \R^d.
\end{equation}
\end{itemize}
\end{remark}
In what follows we want to use $\mu$ in order to construct a new class of \textit{$T \alpha S$ random measures}. For this purpose let $(S,\Sigma,\nu)$ be a $\sigma$-finite measure space and $\mathcal{S}:=\{A \in \Sigma: \nu(A)< \infty\}$. Note that $\mathcal{S}$ is a so-called \textit{$\delta$-ring} whose generated $\sigma$-algebra equals $\Sigma$ (see \cite{integral}). Recall that a mapping $\mathbb{M}:\mathcal{S}\rightarrow \{X:\Omega \rightarrow \R^d \,|  \,\text{$X$ random vector}\}$ is called an $\R^d$-valued iindependently scattered) random measure on $\mathcal{S}$ if the following conditions hold true:
\begin{itemize}
\item[($RM_1$)] For every finite choice $A_1,...,A_k$ of disjoint sets in $\mathcal{S}$ the random vectors $\mathbb{M}(A_1),...,\mathbb{M}(A_k)$ are stochastically independent.
\item[($RM_2$)] For every sequence $(A_n) \subset \mathcal{S}$ of disjoint sets fulfilling $\cup_{n=1}^{\infty} A_n \in \mathcal{S}$ we have that
\begin{equation*}
\mathbb{M} (\cup_{n=1}^{\infty} A_n)=\summe{n=1}{\infty} \mathbb{M}(A_n) \quad \text{almost surely (a.s.).}
\end{equation*}
\end{itemize}
Using Example 3.7 in \cite{integral} we get the following result.
\begin{theorem} \label{15052002}
Recall that $\mu \sim [a,0,\phi]$, let $\mathcal{S}$ and $\nu$ be as before. Then there exists an $\R^d$-valued ISRM $\mathbb{M}$ on $\mathcal{S}$ such that $\mathbb{M}(A)$ has an infinitely-divisible distribution given by $\mathbb{M}(A) \sim [\nu(A) \cdot a, 0, \nu(A) \cdot \phi]$ for every $A \in \mathcal{S}$.
\end{theorem}
In view of \cref{15052001} below it is reasonable to call $\mathbb{M}$ a tempered ($\alpha$-)stable ISRM, which is \textit{generated} by $\mu$. Also note that $\mathbb{M}$, being infinitely-divisible, is closely related to its so-called \textit{control measure} $\lambda_{\mathbb{M}}$ on $\Sigma$, cf. Theorem 3.2 in \cite{integral}. In particular, $\lambda_{\mathbb{M}}$ equals $\nu$ up to a positive constant.  
\subsection{Integration theory}
Usually, the main benefit of ISRMs results from the consideration of corresponding stochastic integrals. More precisely, for $f:S \rightarrow L(\R^d)$, we want to define integrals of the form 
\begin{equation*}
I_{\mathbb{M}}(f):=I(f):=\int_S f(s) \mathbb{M}(ds), 
\end{equation*}
leading to $\R^d$-valued random vectors again. Since this integral is achieved by a stochastic limit, conditions for the existence of $I(f)$ have been studied in \cite{integral} (see \cite{RajRo89} for the case $d=1$). We omit the details. Instead we focus on the main results, adapted to our setting. Therefore, denote the set of all functions $f:S \rightarrow L(\R^d)$ for which $I(f)$ is well-defined by $\mathcal{I}(\mathbb{M})$. Then Proposition 4.5 in \cite{integral} and Theorem 2.3 in \cite{paper2} yield the following.
\begin{theorem} \label{16052050}
Let $\mathbb{M}$ be as in \cref{15052002} and define $U: L(\R^d) \rightarrow \R^d$ by
\begin{equation*}
U(D):=D a + \int_{\R^d} \left( \frac{D x}{1+\norm{D x}^2} -\frac{D x}{1+\norm{x}^2} \right) \,\phi(dx) , \quad D \in L(\R^d).
\end{equation*}
Then, for $f:S \rightarrow L(\R^d)$ measurable, the following statements are equivalent:
\begin{itemize}
\item[(i)] $f \in \mathcal{I}(\mathbb{M})$.
\item[(ii)] $J_1(f)+J_2(f)< \infty$ holds true, where
\begin{equation} \label{eq:22052073}
J_1(f):= \int_S \norm{U(f(s))} \,\nu(ds) \quad \text{and} \quad J_2(f):=\int_S \int_{\R^d}  (1 \wedge \norm{f(s)x}^2) \, \phi(dx) \, \nu(ds).
\end{equation}
\item[(iii)] For every $u \in \R^d$ we have that $\int |\psi(f(s)^{*}u)| \, \nu (ds)< \infty $ and the mapping 
\begin{equation*}
\R^d \ni u\mapsto \int_S \int_{\R^d}(\cos \skp{f(s)^{*}u}{x}-1) \, \phi(dx) \, \nu(ds)
\end{equation*}
is continuous. 
\end{itemize}
In each case it follows that $I(f)$ is infinitely-divisible. More precisely, $I(f) \sim [a^f,0, \phi^f]$, where $a^f:=\int U(f(s)) \, \nu(ds)$ and
\begin{equation} \label{eq:16052001}
 \phi^f(A):= (\nu \otimes \phi) (\{(s,x) \in S \times \R^d: f(s)x \in A \setminus \{0\}\}), \quad A \in \B(\R^d).
\end{equation}
Hence, the LCF of $I(f)$ is given by  $\R^d \ni u \mapsto \int \psi(f(s)^{*}u) \, \nu (ds)$.
\end{theorem}
The notion \textit{stochastic integral} is first of all justified by the fact that $\mathcal{I}(\mathbb{M})$ is a real vector space and that the mapping $\mathcal{I}(\mathbb{M}) \ni f \mapsto I(f)$ is linear. Moreover, for $f,f_1,f_2,... \in \mathcal{I}(\mathbb{M})$, we have that $I(f_n) \rightarrow I(f)$ in probability if and only if 
\begin{equation} \label{eq:19052001}
\forall u \in \R^d: \quad \int_S \psi((f_n(s)-f(s))^{*}u) \, \nu (ds) \rightarrow 0 \quad (\text{as $n \rightarrow \infty$}).
\end{equation}
This and further properties can be found in Proposition 4.3 and Theorem 4.4 in \cite{integral}. In addition, we identify elements in $\mathcal{I}(\mathbb{M})$ that are identical $\nu$-a.e. because this implies that the resulting integrals are equal a.s. and vice versa. \\
Next we want to understand $J_2(f)$ (recall \eqref{eq:22052073}) in terms of the Rosinski measure $\mathcal{R}$, which is inspired by Lemma 2.10 in \cite{rosinski}. 
\begin{lemma} \label{15052076}
We have that 
\begin{equation*} 
J_2(f)< \infty \quad \Leftrightarrow \quad H(f):=\int_S \int_{\R^d} (\norm{f(s)x}^{\alpha} \wedge \norm{f(s)x}^2) \, \mathcal{R}(dx) \, \nu(ds)<\infty. 
\end{equation*}
\end{lemma}
\begin{proof}
Recall that $0<\alpha<2$ is fixed and consider for $z>0$ the incomplete gamma type functions
\begin{equation*}
\gamma(2-\alpha,z):=\int_{0}^{z}u^{(2-\alpha)-1}e^{-u} \, du =\int_{0}^{z}u^{1-\alpha }e^{-u} \, du  \quad \text{and} \quad  \Gamma(- \alpha,z):=\int_{z}^{\infty}u^{- \alpha-1}e^{-u} \,du,
\end{equation*}
respectively. For instance, see Proposition 12 in \cite{jameson} together with \cite{temme} in order to verify that the following asymptotics hold true: As $z \rightarrow 0$ we have that
\begin{equation*} 
z^{\alpha - 2} \gamma(2- \alpha ,z) \rightarrow (2- \alpha)^{-1} \quad \text{and} \quad z^{ \alpha } \Gamma(-\alpha. z)  \rightarrow \alpha^{-1}.
\end{equation*}
On the other hand, as $z \rightarrow \infty$, 
\begin{equation*} 
 \gamma(2- \alpha,z) \rightarrow \Gamma(2-\alpha)=\int_{0}^{\infty} u^{1- \alpha}e^{-u} \,du \quad \text{and} \quad  z^{1+\alpha} e^{z} \Gamma(-\alpha,z)  \rightarrow 1.
\end{equation*}
If we define $g(z):=z^{-2} \gamma(2- \alpha, z)+ \Gamma(- \alpha,z)$, it follows that
\begin{equation} \label{eq:15052093}
\limes{z}{0} z^{\alpha} g(z)= (2-\alpha)^{-1} + \alpha^{-1}>0 ,
\end{equation}
\begin{equation} \label{eq:15052094}
 \limes{z}{\infty} z^2 g(z)= \limes{z}{\infty} ( \gamma(2-\alpha,z) + z^{1+\alpha}e^{z} \Gamma(-\alpha,z) \times z^{1-\alpha}e^{-z}  )=\Gamma(2-\alpha)>0 .
\end{equation}
Combine \eqref{eq:15052093}-\eqref{eq:15052094} and use the fact that $z \mapsto g(z)/(z^{- \alpha} \wedge z^{-2})$ is continuous as well as strictly positive on $(0,\infty)$ to obtain constants $C_1,C_2>0$ fulfilling 
\begin{equation} \label{eq:15052097}
\forall z>0: \quad C_1 (z^{- \alpha} \wedge z^{-2} )\le g(z) \le C_2 (z^{- \alpha} \wedge z^{-2}).
\end{equation}
On the other hand, \eqref{eq:14052002} yields that
\begin{equation*}
J_2(f)=\int_S \int_{\R^d} \int_0^{\infty} (1 \wedge \norm{r f(s)x}^2) r^{-\alpha-1} e^{-r} \, dr \, \mathcal{R}(dx) \, \nu(ds).
\end{equation*}
Now fix $(s,x) \in S \times \R^d$ such that $f(s)x \ne 0$ and let $z=\norm{f(s)x}^{-1}$ to observe that
\begin{equation*}
\int_0^{\infty} (1 \wedge \norm{r f(s)x}^2) r^{-\alpha-1} e^{-r} \, dr =z^{-2} \int_{0}^{z} r^{1- \alpha} e^{-r} \, dr + \int_{z}^{\infty} r^{-\alpha-1} e^{-r} \, dr =g(z).
\end{equation*}
Then the assertion follows from \eqref{eq:15052097}.
\end{proof}
The function $H(f)$, which we just introduced, is similar to the approach in \cite{multi} and will be useful again later. But before we still have to explain in which sense $\mathbb{M}$ and the corresponding integrals have a tempered stable behavior. In this context, note that $f= \indikatorzwei{A} I $ belongs to $\mathcal{I}(\mathbb{M})$ for every $A \in \mathcal{S}$. 
\begin{prop} \label{15052001}
Consider $f,f_1,...f_k \in \mathcal{I}(\mathbb{M})$. Then the random vector $(I(f_1),...,I(f_k))^t$ has a $T \alpha S$ distribution. Moreover, the Rosinski measure of $I(f)$ is given by
\begin{equation} \label{eq:16052004}
\mathcal{R}^f(A):=(\nu \otimes \mathcal{R}) (\{(s,x) \in S \times \R^d: f(s)x \in A \setminus \{0\}\}), \quad A \in \B(\R^d),
\end{equation}
which means that 
\begin{equation}  \label{eq:16052007}
\phi^f(A) = \int_{\R^d} \int_{0}^{\infty}  \indikator{A}{ry} r^{- \alpha-1} e^{-r} \, dr \, \mathcal{R}^f(dy).
\end{equation}
\end{prop}
\begin{proof}
We first assume that $k=1$. Recall that $I(f)$ has no Gaussian part and that its L\'{e}vy measure $\phi^f$ is given by \eqref{eq:16052001}. Using \eqref{eq:14052002} together with the definition of $\mathcal{R}^f$ from \eqref{eq:16052004} it follows for every $A \in \B(\R^d)$ that
\begin{align*}
\phi^f(A)&= \int_S \int_{\R^d} \int_{0}^{\infty} \indikator{A \setminus \{0\}}{r f(s)x} r^{-\alpha-1}e^{-r} \, dr \, \mathcal{R}(dx) \, \nu(ds) \\
&= \int_{\R^d} \int_{0}^{\infty} \indikator{A }{ry} r^{- \alpha -1} e^{-r} \, dr \, \mathcal{R}^f(dy), 
\end{align*}
i.e. \eqref{eq:16052007} holds true. It remains to show that 
\begin{equation} \label{eq:16052027}
\int (\norm{y}^{\alpha} \wedge \norm{y}^2) \, \mathcal{R}^f(dy)< \infty.
\end{equation} 
Because in this case and based on \eqref{eq:16052007}, Theorem 2.3 in \cite{rosinski} states that $\phi^f$ really is the L\'{e}vy measure of a $T \alpha S$ distribution on $\R^d$. Actually, by \eqref{eq:16052004} and in view of \cref{15052076}, condition \eqref{eq:16052027} is equivalent to the finiteness of $J_2(f)$. And since $f \in \mathcal{I}(\mathbb{M})$ implies that $J_2(f)< \infty$ due to \cref{16052050}, this completes the proof for $k=1$. In the general case Corollary 4.9 in \cite{integral} states that $(I(f_1),...,I(f_k))^t$ is infinitely-divisible, while a combination of Example 3.7 and Theorem 4.4 in \cite{integral} shows that its LCF is given by
\begin{equation} \label{eq:17052001}
\R^{k \cdot d} \ni (u_1,...,u_k)^t \mapsto \int_S \psi \left( \sum_{j=1}^{k} f_j(s)^{*}u_j \right) \, \nu(ds).
\end{equation}
Define $\psi'((u_1,...,u_k)^t):=\psi(\sum_{j=1}^{k} u_j).$ Then one can verify that $\psi'$ is the LCF of a distribution $\mu'$ on $\R^{k \cdot d}$, which is also $T \alpha S$. Since the corresponding notation is somewhat involved, the details are left to the reader.  Anyway, it follows from \cref{15052002} that there exists an $\R^{k \cdot d}$-valued ISRM, which is generated by $\mu'$ and which we denote by $\mathbb{M}'$. Moreover, define the block diagonal matrices $f'(s):=f_1(s) \oplus \cdots \oplus f_k(s) \in L(\R^{k \cdot d})$ for every $s \in S$. Then, using \eqref{eq:17052001}, it is easy to check that $f' \in \mathcal{I}(\mathbb{M}')$ and that
\begin{equation*}
(I(f_1),...,I(f_k))^t \overset{d}{=} \int_S f'(s) \, \mathbb{M}'(ds),
\end{equation*}
where $\overset{d}{=}$ means equality in distribution. Hence, the argument for $(I(f_1),...,I(f_k))^t$ having a $T \alpha S$ distribution reduces to the case $k=1$ from above.
\end{proof}
If the $T \alpha S$ distributed generator $\mu$ is proper things become more practicable. For this purpose let $\mathcal{L}^{1}(\nu):=\{g:S \rightarrow \R \, | \, \text{$g$ is measurable and $\int |g(s)| \, \nu(ds)< \infty$}  \}$ and recall \cref{15052055} (i).
\begin{cor} \label{16052075}
Assume that $\mu$ is proper. Then we have:
\begin{itemize}
\item[(i)] $\norm{f}^{\alpha} \in \mathcal{L}^{1}(\nu)$ implies that $J_2(f)< \infty$ and, provided that $f \in \mathcal{I}(\mathbb{M})$, also that the $T \alpha S$ random vector $I(f)$ is proper. 
\item[(ii)] Moreover, if $\int_{\norm{x} \ge 1} \norm{x}^2 \, \mathcal{R}(dx)< \infty$, a sufficient condition for $J_2(f)< \infty$ is given by $(\norm{f}^{\alpha} \wedge \norm{f}^{2}) \in \mathcal{L}^{1}(\nu)$.
\end{itemize}
\end{cor}
\begin{proof}
The fact that $\mu$ is proper yields $\int \norm{x}^{\alpha} \, \mathcal{R}(dx)< \infty$. Recall from \cref{15052076} that $J_2(f)< \infty$ if and only if $H(f)=\int (\norm{y}^{\alpha} \wedge \norm{y}^2) \, \mathcal{R}^f(dy)< \infty$, where $H(f)$ is obviously bounded by
\begin{equation*}
\int_{\R^d} \norm{y}^{\alpha} \, \mathcal{R}^f(dy)=\int_S \int_{\R^d} \norm{f(s)x}^{\alpha} \, \mathcal{R}(dx) \, \nu (ds) \le \int_S \norm{f(s)}^{\alpha} \, \nu(ds)   \int_{\R^d} \norm{x}^{\alpha} \, \mathcal{R}(dx).
\end{equation*}
Hence, (i) follows from \cref{15052001}, since we have that $I(f)$ is proper if and only if $\int \norm{y}^{\alpha} \, \mathcal{R}^f(dy)< \infty$. In a similar way we obtain (ii) due to the following computation:
\begin{align}
H(f) &= \int_S \int_{\R^d} (\norm{f(s)x}^{\alpha} \wedge \norm{f(s)x}^2) \, \mathcal{R}(dx) \, \nu(ds) \notag \\
& \quad \le \int_S (\norm{f(s)}^{\alpha} \wedge \norm{f(s)}^2) \, \nu(ds)  \int_{\R^d} (\norm{x}^{\alpha} \vee \norm{x}^2) \, \mathcal{R}(dx) \notag \\
& \quad = \int_S (\norm{f(s)}^{\alpha} \wedge \norm{f(s)}^2) \, \nu(ds)  \left(   \int_{\norm{x} \le 1 } \norm{x}^{\alpha}  \, \mathcal{R}(dx)  +\int_{\norm{x} \ge 1 } \norm{x}^{2}  \, \mathcal{R}(dx)  \right). \label{eq:21072001}
\end{align}
\end{proof}
Note that the finiteness of $\int_{\norm{x} \ge 1} \norm{x}^2 \, \mathcal{R}(dx)$ is equivalent to the finiteness of $\int_{\norm{x} \ge 1} \norm{x}^2 \, \phi(dx)$ or $\int_{\R^d} \norm{x}^2 \, \mu(dx)$, respectively (see Proposition 2.7 in \cite{rosinski}). 
\subsection{The full and proper-symmetric case} From now on we only consider the case that the generator $\mu$ is full as well as proper and symmetric. In view of \cref{15052055} this implies that $f \in \mathcal{I}(\mathbb{M}) \Leftrightarrow J_2(f)< \infty$ together with $I(f) \sim [0,0, \phi^f]$. In particular, the conditions of \cref{16052075} (i) are fulfilled. The following example fits into this setting and also illustrates part (ii) of \cref{16052075}. 
\begin{example} \label{17052077}
Fix $\lambda>0$ and explicitly consider the proper tempering function $q(r,\theta)=q_{\lambda}(r):=e^{- \lambda r}$ in \eqref{eq:17052010}. Then, if $\sigma$ is symmetric, $\phi=\phi_{\lambda}$ and $\mu_{\lambda} \sim [0,0,\phi_{\lambda}]$ are also symmetric. Moreover, using the construction from (2.3)-(2.5) in \cite{rosinski}, we observe that the corresponding Rosinski measure $\mathcal{R}_{\lambda}$ is concentrated on the set $\lambda \cdot S^{d-1}$, where it equals the image measure $\lambda (\sigma)$, defined by $\lambda (\sigma)(A):=\sigma(\lambda^{-1} \cdot A)$ (cf. Example 1 in \cite{rosinski}). It follows that
\begin{equation*}
\int_{\norm{x} \ge 1} \norm{x}^2 \, \mathcal{R}_{\lambda}(dx) = \int_{S^{d-1}} \norm{\lambda \theta}^2 \, \sigma(d \theta) = \lambda^2 \sigma(S^{d-1})<\infty.
\end{equation*}
Thus, in view of \cref{16052075} (ii), $(\norm{f}^{\alpha} \wedge \norm{f}^{2}) \in \mathcal{L}^{1}(\nu)$ implies that $f \in \mathcal{I}(\mathbb{M}_{\lambda})$, where $\mathbb{M}_{\lambda}$ denotes the ISRM that is generated by $\mu_{\lambda}$. In a similar way it follows from \cref{15052076} that $\mathcal{I}(\mathbb{M}_{\lambda})$ does not depend on $\lambda>0$.
\end{example}
We remark that the previous example will be very important later. Now consider $f,f_1,f_2,...\in \mathcal{I}(\mathbb{M})$ and observe that $I(f_n)-I(f)=I(f_n-f) \sim [0,0,\phi^{f_n-f}]$ in the present case. Hence, as an alternative to criterion \eqref{eq:19052001}, we have that $I(f_n) \rightarrow I(f)$ in probability if and only if the convergence $\phi^{f_n-f} \rightarrow 0$ in the space $\mathcal{M}$ (that is the space of all measures on $\R^d$ that assign finite measure to sets bounded away from zero) holds true. That means $\phi^{f_n-f}(A) \rightarrow 0$ for all $A \in \B(\R^d)$ bounded away from zero. See Theorem 3.1.16 in \cite{thebook} for details. Lemma 2.1 in \cite{integral} shows that this is also  equivalent to $J_2(f_n-f) \rightarrow 0$. On the other hand, relation \eqref{eq:15052097} from the proof of \cref{15052076} revealed that 
\begin{equation} \label{eq:21052011}
C_1 H(f)  \le J_2(f) \le C_2 H(f) 
\end{equation}
holds true for all measurable $f$. Hence, as $n \rightarrow \infty$ we have that
\begin{equation} \label{eq:22052047}
I(f_n) \rightarrow I(f) \quad \text{in probability} \quad \quad \Leftrightarrow \quad H(f_n-f) \rightarrow 0.
\end{equation}
Similarly to (2.13) in \cite{multi} this motivates to introduce
\begin{equation*}
h(D): =\int_{\R^d} (\norm{D x}^{\alpha} \wedge \norm{D x}^2) \, \mathcal{R}(dx), \quad D \in L(\R^d)
\end{equation*}
together with
\begin{equation*}
H(f, \delta):= \int_S h(\delta^{-1} f(s))\, \nu (ds) \in [0,\infty], \quad \delta>0.
\end{equation*}
Recall that $\mathcal{R}$ is $\mathcal{M}$-full due to \cref{15052055}. Then we obtain the following observation. 
\begin{lemma} \label{21052002}
Under the previous assumptions there exists a constant $K>0$ such that $(\norm{D}^{\alpha} \wedge \norm{D}^{2} ) \le K \, h(D)$ holds true for every $D \in L(\R^d)$. 
\end{lemma}
\begin{proof}
Fix $1 \le i,j \le d$. Denote the corresponding entry of the matrix $D$ by $D_{i,j}$ and let $e_j$ be the $j$-th unit vector in $\R^d$. Since $\mathcal{R}$ is $\mathcal{M}$-full, we have that $\mathcal{R}(\{ z e_j : z \ne 0\})>0$. In particular, using that $\mathcal{R}$ is a L\'{e}vy measure, there exists some $0<c_j \le 1$ such that $F_j:=\{ z e_j: |z| \ge  \sqrt{c_j} \}$ fulfills $0<\mathcal{R}(F_j)< \infty$. In view of $c_j^{1/ \alpha} < \sqrt{c_j}$ it follows that 
\begin{align*}
|D_{i,j}|^{\alpha} \wedge D_{i,j}^{2}   & \le \norm{D e_j}^{\alpha} \wedge \norm{D e_j}^{2}  \\
& = (c_j \mathcal{R}(F_j))^{-1} \int ( \norm{c_j^{1/ \alpha} D e_j}^{\alpha} \wedge  \norm{\sqrt{c_j}  D e_j}^{2} ) \indikator{F_j}{x} \, \mathcal{R}(dx) \\
& \le (c_j \mathcal{R}(F_j))^{-1} \int ( \norm{D x }^{\alpha} \wedge  \norm{  D x}^{2} ) \indikator{F_j}{x} \, \mathcal{R}(dx)  \le  (c_j \mathcal{R}(F_j))^{-1} h(D).
\end{align*}
Note that $i$ as well as $j$ were arbitrary and that all norms on $L(\R^d)$ are equivalent. Hence, up to a positive constant, $K$ can be chosen as $\max(\{(c_j \mathcal{R}(F_j))^{-1} : 1 \le j \le d\})$.
\end{proof}
We call a mapping $\norm{\cdot}_{\mathbb{M}}: \mathcal{I}(\mathbb{M}) \rightarrow [0,\infty)$ a \textit{quasi-norm} (on $\mathcal{I}(\mathbb{M})$) if it has the usual properties of a norm, except a possible weakening of the triangular inequality. That is, there exists some $A \ge 1$ such that 
\begin{equation}  \label{eq:21052003}
\forall f_1,f_2 \in \mathcal{I}(\mathbb{M}): \quad \norm{f_1+f_2}_{\mathbb{M}} \le A (\norm{f_1}_{\mathbb{M}}+\norm{f_2}_{\mathbb{M}})
\end{equation}
holds true. Recall that we identify elements in $ \mathcal{I}(\mathbb{M})$ that are identical $\nu$-a.e. 
\begin{theorem} \label{16072090}
\begin{itemize}
\item[(a)] The following identities hold true: 
\begin{align}
\mathcal{I}(\mathbb{M})  &= \{f:S \rightarrow \Li{d} \, | \, \text{$f$ is measurable and $H(f,\delta)< \infty$ for all $\delta>0$} \}  \notag \\
&= \{f:S \rightarrow \Li{d} \, | \, \text{$f$ is measurable and $H(f,\delta)< \infty$ for some $\delta>0$} \}. \notag
\end{align}
\item[(b)] $\norm{f}_{\mathbb{M}}:= \inf \{\delta>0: H(f,\delta) \le 1 \}$ defines a quasi-norm on $\mathcal{I}(\mathbb{M})$.
Moreover, for every $f \in \mathcal{I}(\mathbb{M})$ and $\delta>0 $, we have
\begin{equation} \label{eq:21052006} 
 \left (\frac{\norm{f}_{\mathbb{M}}}{\delta} \right)^{\alpha}  \wedge \left (\frac{\norm{f}_{\mathbb{M}}}{\delta} \right)^2 \le  H(f, \delta) \le  \left (\frac{\norm{f}_{\mathbb{M}}}{\delta} \right)^{\alpha}  \vee \left (\frac{\norm{f}_{\mathbb{M}}}{\delta} \right)^2.
\end{equation}
\item[(c)] There exists some $L_1 \ge 1$ such that, for every $f \in \mathcal{I}(\mathbb{M})$ and $\delta>0$, we have
\begin{equation} \label{eq:22052001} 
\Pro(\norm{I(f)} \ge \delta) \le L_1 \, H(f,\delta).
\end{equation}
Additionally, if $0< p<\alpha$, there exists a constant $L_2=L_2(p)>0$ such that
\begin{equation} \label{eq:21052096}
\erwartung{\norm{I(f)}^p} \le L_2 \norm{f}_{\mathbb{M}}^{p}.
\end{equation}
\item[(d)] The vector space $\mathcal{I}(\mathbb{M})$ is complete with respect to $\norm{\cdot}_{\mathbb{M}}$ and, for $f,f_1,... \in \mathcal{I}(\mathbb{M})$, we have the characterization
\begin{equation*} 
I(f_n) \rightarrow I(f) \text{ in probability } \quad \Leftrightarrow \quad  \norm{f_n-f}_{\mathbb{M}} \rightarrow 0.
\end{equation*}
\end{itemize} 
\end{theorem}
\begin{proof}[\bf{Proof of \cref{16072090}}]  \label{proof3}
Note that $H(f)=H(f,1)$ and verify that part (a) follows from \cref{15052076} after a verification of
\begin{equation} \label{eq:21052001}
\forall \delta>0: \quad (\delta^{- \alpha} \wedge \delta^{-2}) H(f,1) \le H(f,\delta) =H (\delta^{-1} f, 1)\le (\delta^{- \alpha} \vee \delta^{-2}) H(f,1).
\end{equation}
We now prove part (b). Obviously, $f = 0$ implies that $H(f,\delta)=0$ for all $\delta>0$, which yields $\norm{f}_{\mathbb{M}}=0$. Conversely, if $\norm{f}_{\mathbb{M}}=0$, we derive from \eqref{eq:21052001} that $H(f,1)=0$, i.e. $h(f(s))=0$ for almost every $s \in S$. In view of \cref{21052002} we obtain for those $s$ that $f(s)=0$. Also note that the we have $H( \rho f, \delta)=H( f, |\rho|^{-1} \delta)$ for all $\rho \ne 0$ and $\delta>0$. Hence, the homogeneity property $\norm{ \rho f}_{\mathbb{M}}= |\rho| \, \norm{ f}_{\mathbb{M}}$ ($\rho \in \R$) holds true by definition of $\norm{\cdot}_{\mathbb{M}}$. It remains to prove \eqref{eq:21052003}. For this purpose recall that $\norm{u_1+u_2}^2 \le 2(\norm{u_1}^2+\norm{u_2}^2)$ ($u_1, u_2 \in \R^d$) and that the mapping $[0,\infty) \ni z \mapsto(1 \wedge z)$ is sub-additive. Using \eqref{eq:21052011} we compute for all $\delta>0$ and $f_1, f_2 \in \mathcal{I}(\mathbb{M})$ that
\begin{align*}
H(f_1+f_2,\delta) & \le C_1^{-1} J_2(\delta^{-1}(f_1+f_2)) \\
& \le 2 C_1^{-1} (J_2(\delta^{-1}f_1)+J_2(\delta^{-1}f_2))  \le 2 C_2/C_1 (H(f_1,\delta)+H(f_2,\delta)).
\end{align*}
Now choose $A\ge 1$ such that $H( f, A \delta) \le (4 C_2 / C_1)^{-1} H(f, \delta)$ holds true for every $\delta>0$ and $f$ as before, which is possible due to \eqref{eq:21052001}. Since $H(f, \cdot)$ is decreasing, it follows for $\eps>0$ arbitrary that 
\begin{align*}
& H(f_1+f_2,A(\norm{f_1}_{\mathbb{M}}  +  \norm{f_2}_{\mathbb{M}}+ \eps))  \\
& \quad \le 2 C_2/C_1 ( H(f_1, A (\norm{f_1}_{\mathbb{M}}+\norm{f_2}_{\mathbb{M}} +\eps)) + H(f_2, A (\norm{f_1}_{\mathbb{M}}+\norm{f_2}_{\mathbb{M}} +\eps)) ) \\
&  \quad \le 2 C_2/C_1 ( H(f_1, A (\norm{f_1}_{\mathbb{M}} +\eps)) + H(f_2, A (\norm{f_2}_{\mathbb{M}} +\eps)) ) \\
& \quad \le 2^{-1} (H(f_1,\norm{f_1}_{\mathbb{M}} +\eps )  + H(f_2,\norm{f_2}_{\mathbb{M}} +\eps )  ) \le 1,
\end{align*}
where in the last step we used the definition of $\norm{\cdot}_{\mathbb{M}}$. Letting $\eps \rightarrow 0$, the same argument shows \eqref{eq:21052003} and hence that $\norm{\cdot}_{\mathbb{M}}$ is a quasi-norm. Finally, \eqref{eq:21052006} is obvious in the case $\norm{f}_{\mathbb{M}}=0$. Else we observe that $H(f, \norm{f}_{\mathbb{M}})=1$ by continuity of $\delta \mapsto H(f,\delta)$. Then, similar as before, we can write $H(f,\delta)=H( (\delta/ \norm{f}_{\mathbb{M}})^{-1} f, \norm{f}_{\mathbb{M}})$ such that \eqref{eq:21052006} follows from \eqref{eq:21052001}.  \\
For the proof of part (c) recall that $|1-\exp(z)| \le |z|$ for any $z \in \C$ with $\text{Re }z \le 0$ and that the LCF of $I(f)$ is given by \eqref{eq:17052001} for $k=1$. Combine this with Lemma 2.7 in \cite{multi} to obtain a constant $C>0$ such that, for every $\delta>0$, the relation
\begin{equation} \label{eq:22052056}
\Pro(\norm{I(f)} \ge \delta ) \le C \delta^{d} \int_{\norm{u}_{\infty} \le \delta^{-1}} \int_S | \psi(f(s)^{*}u )| \, \nu(ds) \, du
\end{equation}
holds true. Here $\norm{\cdot}_{\infty}$ denotes the maximum norm on $\R^d$. At the same time there exists some $C_0 \ge 1$ fulfilling $\norm{\cdot} \le C_0 \norm{\cdot}_{\infty}$. Also recall \eqref{eq:14052005} and, for every $z \in \R$, that $|\cos(z)-1| \le 2(1 \wedge z^2)$. From the Cauchy-Schwarz inequality it follows for every $\norm{u}_{\infty} \le \delta^{-1}$ that
\begin{align}
\int_S | \psi(f(s)^{*}u)| \, \nu(s) & \le \int_S \int_{\R^d} |\cos \skp{f(s)^{*}u}{x} -1| \, \phi(dx) \, \nu(s)   \notag   \\
& \le 2 \int_S \int_{\R^d} (1 \wedge \norm{u}^2 \norm{f(s)x}^2) \, \phi(dx)  \, \nu(s) \label{eq:21072077}  \\
& \le 2 \, C_0^2 \int_S \int_{\R^d} (1 \wedge \norm{\delta^{-1}f(s)x}^2) \, \phi(dx) \, \nu(s) = 2 \, C_0^2 \, J_2(\delta^{-1} f). \notag
\end{align}
Let $L_1:=2  C_0^2 \, C \, C_2$ to verify that \eqref{eq:21052011} together with \eqref{eq:22052056} implies \eqref{eq:22052001}. For the second statement of part (c) fix $0<p< \alpha$. Then it is well-known that
\begin{align*}
\erwartung{\norm{I(f)}^p} = p \integral{0}{\infty}{ \delta^{p-1} \, \Pro(\norm{I(f)} \ge \delta) }{d \delta},
\end{align*}
where \eqref{eq:21052006}-\eqref{eq:22052001} provide the inequality 
\begin{equation*}
\Pro(\norm{I(f)} \ge \delta) \le \indikator{[0,\norm{f}_{\mathbb{M}}] }{\delta}  +  L_1 ( \norm{f}_{\mathbb{M}}/ \delta)^{\alpha}  \indikator{(\norm{f}_{\mathbb{M}},\infty ) }{\delta}, \quad  \delta>0.
\end{equation*}
Now, as in the proof of Theorem 2.8 in \cite{multi}, we observe that \eqref{eq:21052096} holds true for a constant $L_2>0$. The details are left to the reader. In order to show that $\mathcal{I}(\mathbb{M})$ is complete, let $(f_n) \subset \mathcal{I}(\mathbb{M})$ be a Cauchy-sequence with respect to $\norm{\cdot}_{\mathbb{M}}$. Then we have to find some $f \in \mathcal{I}(\mathbb{M})$ such that $\norm{f_n-f}_{\mathbb{M}} \rightarrow 0$ as $n \rightarrow \infty$, or at least along a suitable subsequence, since $(f_n)$ is Cauchy and since $\norm{\cdot}_{\mathbb{M}}$ fulfills \eqref{eq:21052003}. Hence, without loss of generality and throughout following the idea of the proof of Theorem 5.2.1 in \cite{Dud02}, it can be assumed that 
\begin{equation} \label{eq:10101803}
\norm{f_m-f_n}_{\mathbb{M}} \le 2^{-N} \quad \text{for all $m,n \ge N$}.
\end{equation}
Recall the definition of $H(\cdot,1)$ and define the sets $A_j:= \{s \in S:  h(f_{j+1}(s)-f_j(s))> j^{-4}  \}$ for all $j \in \N$. Using \eqref{eq:21052006} and \eqref{eq:10101803}, it follows that
\begin{equation*}
j^{-4} \nu(A_n) \le  H(f_{j+1}-f_{j},1) \le  2^{- \alpha j}, \quad j \in \N,
\end{equation*}
which implies that $\summezwei{j=1}{\infty} \nu(A_j)<\infty$. Thus $B:= \limsup_{j \rightarrow \infty} A_j$ is a $\nu$-null set. Let $f(s):=0$ for every $s \in B$. Conversely, for $s \in B^c=S \setminus B$, there exists some $N(s) \in \N$ fulfilling $s \notin A_j$ for all $j \ge N(s)$. Furthermore, \cref{21052002} allows us to assume that $N(s)$ is chosen large enough such that we have $\norm{f_{j+1}(s)-f_j(s)} \le 1$ for every $j \ge N(s)$. Actually, if we use \cref{21052002} again, this implies for all $s \in B^c$ and $m,n \ge N(s)$ that
\allowdisplaybreaks
\begin{align*}
\norm{f_m(s)-f_n(s)} & \le \summe{j=N(s)}{\infty}  \sqrt{ \norm{(f_{j+1}(s)-f_{j}(s))}^2  }\\ 
&=  \summe{j=N(s)}{\infty}  \sqrt{ \norm{(f_{j+1}(s)-f_{j}(s))}^{\alpha} \wedge \norm{(f_{j+1}(s)-f_{j}(s))}^2  }\\
& \le \sqrt{K} \summe{j=N(s)}{\infty}  \sqrt{  h(f_{j+1}(s)-f_{j}(s))   }  \le \sqrt{K} \summe{j=N(s)}{\infty}  j^{-2} .
\end{align*}
It follows that $\norm{f_m(s)-f_n(s)}  \rightarrow 0$ as $N(s)$ tends to $\infty$, which shows that $(f_n(s))$ is Cauchy with respect to the operator norm $\norm{\cdot}$. Denote the corresponding limit by $f(s)$ and observe that $f:S \rightarrow L(\R^d)$ is measurable with $f_n(s) \rightarrow f(s)$ $\nu$-almost everywhere as $n \rightarrow \infty$. Moreover, \eqref{eq:21052003} and \eqref{eq:10101803} imply for every $n \in \N$ that $\norm{f_n}_{\mathbb{M}} \le A(\norm{f_1}_{\mathbb{M}}+2^{-1}).$ Also note that $D \mapsto h(D)$ is continuous by the dominated convergence theorem. Hence, Fatou's Lemma and \eqref{eq:21052006} yield that
\allowdisplaybreaks 
\begin{align*}
H(f,1) = \int_S \liminf_{n \rightarrow \infty}  h(f_n(s)) \, \nu(ds) \le  \limesinf{n}{\infty} H(f_n,1) \le   \limesinf{n}{\infty} ( \norm{f_n}_{\mathbb{M}}^{\alpha} \vee \norm{f_n}_{\mathbb{M}}^2 ) < \infty,
\end{align*}
i.e. $f$ belongs to $\mathcal{I}(\mathbb{M})$ due to part (a). Similarly, it follows from \eqref{eq:10101803} for every $n \in \N$ that 
\begin{equation*}
H(f_n-f,1)  \le  \limesinf{m}{\infty} ( \norm{f_n-f_m}_{\mathbb{M}}^{\alpha} \vee \norm{f_m-f_m}_{\mathbb{M}}^2 )  \le  \,  2^{- \alpha n}.
\end{equation*}
In view of \eqref{eq:21052006} this shows that $\norm{f_n-f}_{\mathbb{M}} \rightarrow 0$ as $n \rightarrow \infty$. The additional statement of part (d) is just a consequence of \eqref{eq:22052047} and \eqref{eq:21052006}.
\end{proof}
\begin{remark}
Since $I(f)$ has a $T \alpha S$ distribution it is well-known that $\erwartung{\norm{I(f)}^{p}}$ is finite for all $0<p< \alpha$. However, this can also be true for $p \ge \alpha$, depending on the behavior of $\mathcal{R}^{f}$ (see Proposition 2.7 in \cite{rosinski}). But in this case we would need sharper estimates in order to perform an appropriate proof of \eqref{eq:21052096}.
\end{remark}
We finish this section with a useful observation, which, for $D=f(s)$, we implicitly encountered within the foregoing proofs. More precisely, recall \eqref{eq:14052002} and \eqref{eq:21072077} together with the proofs of \cref{15052076} and \cref{16052075} (particularly \eqref{eq:21072001}) as well as \cref{17052077}. 
\begin{cor} \label{21072080}
In the situation of \cref{17052077} we denote the LCF of $\mu_{\lambda}$ by $\psi_{\lambda}$. Then there exists a constant $T>0$ such that
\begin{equation*}
 |\psi_{\lambda}(D^{*}u)|  \le T (1+ \norm{u}^2) (\norm{D}^{\alpha} \wedge \norm{D}^2)
\end{equation*}
holds true for every $D \in L(\R^d)$ and $u \in \R^d$.
\end{cor}
%% Section 3: 
\section{Examples of multivariate tempered stable random fields: Moving-average and harmonizable representation} \label{chapter3}
Throughout this section fix $\lambda>0$ and, for the rest of this paper, we explicitly consider the $n$-dimensional Lebesgue measure $\nu(ds)=ds$ on $(S, \Sigma)=(\R^n, \B(\R^n))$. Then, until further notice (in \cref{subsec}), let $\mu_{\lambda}$ and $\mathbb{M}_{\lambda}$ be as in \cref{17052077}, where $\sigma$ is a finite, symmetric measure on $S^{d-1}$. In addition, we assume that $\sigma$ is $\mathcal{M}$-full, since this ensures that the generator $\mu_{\lambda}$ is full.  
\begin{example}
Consider the case $n=1$ and let $X(t):=\mathbb{M}_{\lambda}([0,t])$ for every $t \ge 0$. If we recall $(RM_1)$-$(RM_2)$ it is easy to see that the resulting $\R^d$-valued stochastic process $\mathbb{X}=\{X(t): t \ge 0 \}$ is a L\'{e}vy process. And since $X(1)$ has a $T \alpha S$ distribution it follows that $\mathbb{X}$ is a \textit{$T \alpha S$ L\'{e}vy process} in the sense of \cite{rosinski}.
\end{example}
Note that there are interesting results in \cite{rosinski} concerning the so-called \textit{short time} and \textit{long time behavior} of $T \alpha S$ L\'{e}vy processes, respectively. However, in what follows, we treat the general case $n \ge 1$. More precisely and in order to take advantage of our previous efforts, we will investigate $\R^d$-valued random fields of the form $\mathbb{X}=\{X(t): t \in \R^n\}$, where $X(t)=I(f_t)$ for every $t \in \R^n$ and a suitable integrand (or kernel function) $f_t \in \mathcal{I}(\mathbb{M}_{\lambda})$. In this context, note that we benefit from \cref{15052001} in the following sense, where $0< \alpha<2$ is still fixed.
\begin{defi} \label{14072001}
A random field $\mathbb{X}=\{X(t): t \in \R^n\}$ is called \textit{tempered ($\alpha$-)stable} (abbreviated as a \textit{$T \alpha S$ random field}) if all finite-dimensional distributions are $T \alpha S$.
\end{defi}
\subsection{Moving-average representation}
Our multivariate approach needs some preparation. Fix some matrix $E \in Q(\R^n)$, where 
\begin{equation*}
Q(\R^n):=\{E \in L(\R^n): \text{All real parts of the eigenvalues of $E$ are strictly positive}  \}.
\end{equation*}
Recall that $\exp(E)$, i.e. the \textit{matrix exponential} of $E$ is given by
\begin{equation*}
\exp(E)=\sum_{j=0}^{\infty} \frac{E^j}{j!} \quad \text{with} \quad c^E:= \exp( (\log c) E), \quad c>0.
\end{equation*}
For instance, see Proposition 2.2.2 in \cite{thebook} for basic properties about the matrix exponential that we will use in sequel. This allows to define so-called \textit{generalized polar coordinates with respect to $E$}. More precisely, Lemma 6.1.5 in \cite{thebook} states that there exists a norm $\norm{\cdot}_E$ on $\R^n$ such that for $S_E^{n-1}:= \{x \in \R^n: \norm{x}_E=1 \}$ the mapping $\Psi: (0, \infty) \times S_E^{n-1} \rightarrow \R^n \setminus \{0\}$, $\Psi(c,\theta)=c^E \theta$ is a homeomorphism. In other words, every $x \ne 0$ can be uniquely written as $x= \tau(x)^{E} l(x)$ for some \textit{radial part} $\tau(x)>0$ and some \textit{direction} $l(x) \in S_E^{n-1}. $ In this context, a function $\varphi: \R^n \rightarrow \C$ is called \textit{$E$-homogeneous} if $\varphi(c^E x)=c \, \varphi (x)$ for all $c>0$ and $x \ne 0$. On the other hand, letting $\beta>0$, the authors in \cite{bms} say that a function $\varphi: \R^n \rightarrow [0, \infty)$ is \textit{$(\beta,E)$-admissible} if the following conditions hold true:
\begin{itemize}
\item[(i)] $\varphi(x)>0$ for all $x \ne 0$.
\item[(ii)] $\varphi$ is continuous and for any $0<A<B$ there exists a constant $C>0$ such that, for $A \le \norm{y} \le B$,
\begin{equation} \label{eq:15072065}
\tau(x) \le 1 \quad \Rightarrow \quad | \varphi(x+y) - \varphi(y)| \le C \tau(x)^{\beta}.
\end{equation}
Here, $\tau(x)=\tau_E(x)$ is the radial part of $x$ in terms of the generalized polar coordinates with respect to $E$.
\end{itemize}
Now we are able to state the first main result of this section, where the two last-mentioned properties are actually combined. See Theorem 2.11 and Corollary 2.12 in \cite{bms} for useful examples of such functions. Also recall that $I=I_d$ is the identity operator on $\R^d$.
\begin{theorem} \label{15072001} 
Let $E$ be as before and denote its trace by $q$. Assume that $\varphi: \R^n \rightarrow [0,\infty)$ is an $E$-homogeneous, $(\beta,E)$-admissible function for some $\beta>0$. Moreover, consider a $d \times d$-matrix $D \in Q(\R^d)$ such that the maximum of the real parts of the eigenvalues of $D$, denoted by $H$, fulfills 
\begin{equation} \label{eq:14072005}
H < \beta + q \left ( \frac{1}{\alpha} - \frac{1}{2} \right).
\end{equation}
Then the random field $\mathbb{X}=\{X(t): t \in \R^n\}$, defined by
\begin{equation*}
X(t):= \int_{\R^n} \left[  \varphi(t-s)^{D- \frac{q}{\alpha}  I} - \varphi(-s)^{D-  \frac{q}{\alpha} I}   \right] \, \mathbb{M}_{\lambda}(ds), 
\end{equation*}
exists. Furthermore:
\begin{itemize}
\item[(a)] $\mathbb{X}$ is a $T \alpha S$ random field in the sense of \cref{14072001}.
\item[(b)] $\mathbb{X}$ is stochastically continuous.
\item[(c)] $\mathbb{X}$ has stationary increments. That is, for any $h \in \R^n$,
\begin{equation*}
\{X(t+h)-X(h): t \in \R^n\} \overset{fdd}{=}  \{X(t): t \in \R^n\}  ,
\end{equation*}
where $\overset{fdd}{=}$ denotes equality of all finite-dimensional distributions. 
\item[(d)] If $ \frac{q}{\alpha}$ is not an eigenvalue of $D$, then $X(t)$ is full for every $t \ne 0$.
\end{itemize}
\end{theorem}
\begin{proof}
Note that $X(0)=0$ a.s. Then fix $t \in \R^n \setminus \{0\}$ and observe that $\varphi(x)=0 \Leftrightarrow x=0$. It follows that
\begin{equation} \label{eq:15072001}
f_t(s):=\varphi(t-s)^{D- \frac{q}{\alpha}  I} - \varphi(-s)^{D-  \frac{q}{\alpha} I} , \quad s \in \R^n
\end{equation}
is well-defined for every $s \notin \{0,t\}$. However, the cases $s \in \{0,t\}$ are negligible since, up to a constant, the $n$-dimensional Lebesgue measure equals the control measure of $\mathbb{M}_{\lambda}$  in the present case. Then it suffices to verify that $(\norm{f_t(s)}^{\alpha} \wedge \norm{f_t(s)}^{2}) \in \mathcal{L}^{1}(ds)$ due to \cref{17052077}. For this purpose denote the generalized polar coordinates with respect to $E$ by $(\tau(\cdot),l(\cdot))$ and let $\tau(0)=0$. Recall the proof of Theorem 2.5 in \cite{lixiao} and that, if $H< \beta$ holds true instead of \eqref{eq:14072005}, there it has been shown for some suitably chosen $\gamma >0$ that 
\begin{equation*}
\norm{f_t(s)}^{\alpha} = \norm{f_t(s)}^{\alpha}\indikatorzwei{ \{ \tau(s) \le \gamma \} }   +\norm{f_t(s)}^{\alpha} \indikatorzwei{ \{  \tau(s) >\gamma  \} }  \, \,  \in \mathcal{L}^{1}(ds).
\end{equation*}
Actually, for the verification of $\norm{f_t(s)}^{\alpha} \indikatorzwei{ \{\tau(s) \le \gamma\} } \in \mathcal{L}^{1}(ds)$ the quoted proof only makes use of the assumptions on $\varphi$ together with $D \in Q(\R^d)$. Conversely, it uses $H< \beta$ in order to establish the accuracy of $\norm{f_t(s)}^{\alpha} \indikatorzwei{ \{\tau(s)>\gamma\} } \in \mathcal{L}^{1}(ds)$. Now it is easy to verify that condition \eqref{eq:14072005} allows us to slightly modify the corresponding estimates from the proof of Theorem 2.5 in \cite{lixiao} and hence to obtain that $\norm{f_t(s)}^2 \indikatorzwei{ \{\tau(s)>\gamma\} } \in \mathcal{L}^{1}(ds)$ holds true. The details are left to the reader.  Anyway, in view of
\begin{equation} \label{eq:14072076}
(\norm{f_t(s)}^{\alpha} \wedge \norm{f_t(s)}^{2}) \le \norm{f_t(s)}^{\alpha} \indikatorzwei{  \{ \tau(s) \le \gamma \}  } +\norm{f_t(s)}^{2} \indikatorzwei{  \{ \tau(s) > \gamma \}  }, \quad s \in \R^n,
\end{equation}
this proves the existence of $\mathbb{X}$. Then part (a) follows from \cref{15052001}. Now let us prove part (b). For this purpose fix $t \in \R^n$ and observe that, due to \eqref{eq:22052047}, we have to show that $H(f_{t+h}-f_t) \rightarrow 0$ as $h \rightarrow 0$. Hence, if we recall \cref{17052077} together with \eqref{eq:21072001},
the claimed stochastic continuity would follow if
\begin{equation*} 
\int_{\R^n}  \left (\norm{f_{t+h}(s)-f_t(s)}^{\alpha} \wedge \norm{f_{t+h}(s)-f_t(s)}^{2}  \right )     \, ds \rightarrow 0 
\end{equation*}
holds true as $h \rightarrow 0$. Thus by \eqref{eq:15072001}, a change of variables and \eqref{eq:14072076}, it suffices to prove that
\begin{equation} \label{eq:14072056}
\int_{\R^n}  \left (\norm{f_{h}(s)}^{\alpha} \indikatorzwei{  \{ \tau(s) \le \gamma \}  }  + \norm{f_h(s)}^{2} \indikatorzwei{  \{ \tau(s) > \gamma \}  }  \right )     \, ds \rightarrow 0 \quad (h \rightarrow 0),
\end{equation}
actually leading back to the above-mentioned modification and its corresponding estimates. We omit the details, which, using the dominated convergence theorem, can be found in the proof of Theorem 2.5 in \cite{lixiao} again and which eventually give \eqref{eq:14072056}. Moreover, by linearity of the stochastic integral and \eqref{eq:17052001}, the proof of part (c) merely reduces to the foregoing change of variables. Finally, fix $t \ne 0$ and recall from Remark 2.1 in \cite{lixiao} that $f_t(s)$ is invertible for $s$ in a subset of $\R^n$ with positive Lebesgue measure provided that $ \frac{q}{\alpha}$ is not an eigenvalue of $D$. Then the assertion of part (d) follows from Proposition 2.6 (a) in \cite{paper2}. 
\end{proof}
\begin{remark}
Assume that \eqref{eq:14072005} is particularly fulfilled with $H < \beta$. Then the foregoing proof showed that the function $\norm{f_t(s)}^{\alpha}$ from \eqref{eq:15072001} belongs to $\mathcal{L}^{1}(ds)$. In this case \cref{16052075} (a) implies that $X(t)$ has even a proper $T \alpha S$ distribution.
\end{remark}
\subsection{Harmonizable representation} \label{subsec}
Although \cref{15072001} provides a large class of multivariate tempered stable random fields by now, we want to present a further class of this type in the sequel. The resulting random fields use a so-called \textit{harmonizable representation}, which is also popular in the context of classical $\alpha$-stable settings, that is without the use of tempering. For instance see \cite{bms,lixiao,SaTaq94}, just to mention a few. It is also mentionable that, by investigating the scalar-valued $\alpha$-stable moving-average and harmonizable representation from \cite{bms},  Proposition 6.1 in \cite{BiLa09} already established different path properties between both representations. Note that \cref{21072001} as well as \cref{16072088} below will emphasize these differences.  \\
In what follows we have to use kernel functions taking values in $L(\C^d)$, that is the space of $d \times d$-matrices with entries from the complex numbers $\C$. Moreover, this also requires the use of a complex-valued tempered stable ISRM. Essentially, we do so by 
modifying the underlying generator. More precisely, let $\tilde{\sigma}$ be a symmetric and finite as well as $\mathcal{M}$-full measure on $S^{2d-1}=\{x \in \R^{2d}: \norm{x}=1 \}$. For $\lambda>0$ as before consider the tempering function $q(r, \theta)=q_{\lambda}(r)=e^{- \lambda r}$ again, but based on the polar coordinates in $\R^{2d}$ now. Accordingly to \eqref{eq:17052010} this allows to define a L\'{e}vy measure $\tilde{\phi}_{\lambda}$ on $\R^{2d}$. Let $\tilde{\mu}_{\lambda} \sim [0,0, \tilde{\phi}_{\lambda}]$ with LCF $\tilde{\psi}_{\lambda}$ be the generator of an $\R^{2d}$-valued ISRM in the sense of \cref{15052002} (still with $\nu(ds)=ds$ on $(S, \Sigma)=(\R^n, \B(\R^n))$), denoted by $\widetilde{\mathbb{M}}_{\lambda}$. \\
Note that $\widetilde{\mathbb{M}}_{\lambda}$ can be naturally identified with an $\C^d$-valued ISRM, which we denote by $\mathbb{M}^{\C}_{\lambda}$ hereinafter and which, similarly as before, allows for the use as integrator in a corresponding integration theory. We refer the reader to \cite{integral} for details. In any case, Proposition 4.11 in \cite{integral} states for a measurable function $g:\R^n \rightarrow L(\C^d)$ that the $\R^d$-valued random vector
\begin{equation} \label{eq:15072075}
\text{Re }  \int_{\R^n} g(s) \, \mathbb{M}^{\C}_{\lambda}(ds)
\end{equation}
is well-defined (as a stochastic limit) if and only if  
\begin{equation} \label{eq:150720100}
\tilde{g}:=\begin{pmatrix} \text{Re } g & - \text{Im } g\\ 0 & 0  \end{pmatrix} \in \mathcal{I}(\widetilde{\mathbb{M}}_{\lambda}),
\end{equation}
where $\tilde{g} \in L(\R^{2d})$. Then we can state the second main result of this section. 
\begin{theorem} \label{15012005}
Let $E \in Q(\R^n)$, denoting its trace by $q$ again and its smallest real part of the eigenvalues by $a$. Assume $\varphi: \R^n \rightarrow [0,\infty)$ to be a continuous and $E^{*}$-homogeneous function fulfilling $\varphi(x)>0$ for all $x \in \R^n \setminus \{0\}$. Also consider $D \in L(\R^d)$, where $h$ and $H$ denote the minimum and the maximum of the real parts of the eigenvalues of $D$, respectively. Moreover, let the assumption
\begin{equation} \label{eq:14072099}
 q \left( \frac{1}{2} - \frac{1}{\alpha} \right)<h  \le H < a
\end{equation}
hold true. Then the random field $\mathbb{Y}=\{Y(t): t \in \R^n\}$, defined by
\begin{equation} \label{eq:17072001}
Y(t):= \text{Re }  \int_{\R^n} \left(e^{ \im \skp{t}{s}} -1 \right) \varphi(s)^{-D- \frac{q}{\alpha} I}  \, \mathbb{M}^{\C}_{\lambda}(ds), 
\end{equation}
exists in the sense of \eqref{eq:15072075}. Furthermore:
\begin{itemize}
\item[(a)] $\mathbb{Y}$ is tempered stable.
\item[(b)] $\mathbb{Y}$ is stochastically continuous.
\item[(c)] $Y(t)$ is full for every $t \ne 0$.
\end{itemize}
\end{theorem}
\begin{proof} 
Denote the generalized polar coordinates with respect to $E^t$ by $(\tau_1(\cdot),l_1(\cdot))$ and observe that Lemma 4.2 in \cite{lixiao}, although stated for $D \in Q(\R^d)$, remains true for $D \in L(\R^d)$ due to (2.2) in \cite{paper2}. Then, using the assumption $H<a$, it follows exactly as in the proof of Theorem 2.6 in \cite{lixiao} that $\norm{s}^{\alpha}  \norm{\varphi(s)^{-D- \frac{q}{\alpha} I }}^{\alpha} \indikatorzwei{ \{\tau_1(s) < 1\} }  \in \mathcal{L}^1(ds)$. Conversely, assuming that $h>0$ (which is nothing else than $D \in Q(\R^d))$, the aforementioned proof in \cite{lixiao} showed that $ \norm{\varphi(s)^{-D- \frac{q}{\alpha} I }}^{\alpha} \indikatorzwei{ \{\tau_1(s) \ge  1\} }  \in \mathcal{L}^{1}(ds)$. Actually, by a slight modification it turns out that we just need the assumption $h>q ( \frac{1}{2} - \frac{1}{\alpha})$ from \eqref{eq:14072099} in order to establish that $ \norm{\varphi(s)^{-D- \frac{q}{\alpha} I }}^{2} \indikatorzwei{ \{\tau_1(s) \ge  1\} }  \in \mathcal{L}^{1}(ds)$ instead. If we combine our findings it follows that the function 
\begin{equation} \label{eq:23072005}
\Psi(s):=\norm{s}^{\alpha}  \norm{\varphi(s)^{-D- \frac{q}{\alpha} I }}^{\alpha} \indikatorzwei{ \{\tau_1(s) < 1\} } +   \norm{\varphi(s)^{-D- \frac{q}{\alpha} I }}^{2} \indikatorzwei{ \{\tau_1(s) \ge 1\} }, \quad s \in \R^n
\end{equation}
belongs to $\mathcal{L}^{1}(ds)$, where the case $s=0$ is negligible again. Now fix $t \in \R^n$. Then, in view of \eqref{eq:150720100} and Euler's formula, we define
\begin{equation}  \label{eq:16072047}
\tilde{g}_t(s):=\begin{pmatrix}  (\cos \skp{t}{s}-1) \varphi(s)^{-D- \frac{q}{\alpha} I}  & - \sin\skp{t}{s}\varphi(s)^{-D- \frac{q}{\alpha} I}  \\ 0 & 0  \end{pmatrix}, \quad s \in \R^n
\end{equation}
and observe by equivalence of all norms on $L(\R^{2d})$ that there exists some $K_0 >0$ fulfilling
\begin{equation} \label{eq:04082001}
 \norm{\tilde{g}_t(s)} \le K_0 ( |\cos \skp{t}{s}-1|  + |\sin \skp{t}{s} |  )   \norm{\varphi(s)^{-D- \frac{q}{\alpha} I }}, \quad s \in \R^n.
\end{equation}
Recall the idea from \eqref{eq:14072076} and the inequality $(a+b)^{\alpha} \le 2(a^{\alpha}+b^{\alpha})$. Then we obtain another constant $K_1>0$ such that, for every $s \in \R^n$, $ \norm{\tilde{g}_t(s)}^{\alpha} \wedge  \norm{\tilde{g}_t(s)}^{2} $ is bounded by
\begin{equation*}
K_1 ( |\cos \skp{t}{s}-1|^{\alpha}  + |\sin \skp{t}{s} |^{\alpha}   )   \norm{\varphi(s)^{-D- \frac{q}{\alpha} I }}^{\alpha} \indikatorzwei{ \{\tau_1(s) < 1\} } +   \norm{\varphi(s)^{-D- \frac{q}{\alpha} I }}^{2} \indikatorzwei{ \{\tau_1(s) \ge 1\} }.
\end{equation*}
Finally, using the Cauchy-Schwarz inequality together with some routine estimates for $\sin(\cdot)$ and $\cos(\cdot)$, we verify that 
\begin{align} \label{eq:2207201000}
 \norm{\tilde{g}_t(s)}^{\alpha} \wedge  \norm{\tilde{g}_t(s)}^{2}  & \le K_2  (1+ \norm{t}^{\alpha}) \Psi(s), \quad s \in \R^n
\end{align}
holds true for some constant $K_2 >0$ and due to \eqref{eq:23072005}. In view of \eqref{eq:150720100} and \cref{17052077} (which remains true accordingly) this proves the existence of $\mathbb{Y}$. We now prove part (a) and consider $k \in \N$ as well as $t_1,...,t_k \in \R^n$ arbitrary. It follows from Remark 4.13 in \cite{integral} that the random vector $(Y(t_1),...,Y(t_k))^t$ is infinitely-divisible and that its LCF is given by
\begin{equation}  \label{eq:16072001}
u \mapsto \int_{\R^n} \tilde{\psi}_{\lambda} \left( \left( \sum_{j=1}^{k} (\cos \skp{t_j}{s}-1) \varphi(s)^{-D^{*}- \frac{q}{\alpha} I}u_j, - \sum_{j=1}^{k} \sin \skp{t_j}{s}  \varphi(s)^{-D^{*}- \frac{q}{\alpha} I}u_j \right )^t \right) \, ds,
\end{equation}
where $u=(u_1,...,u_k)^t \in \R^{k \cdot d}$. Recall the specific form of $\tilde{g}_t$ from \eqref{eq:16072047} and observe that $(Y(t_1),...,Y(t_k))^t$ can be regarded as projection of the $\R^{k \cdot (2d)}$-valued random vector $(Z(t_1),...,Z(t_k))^t$, where $Z(t):= \int \tilde{g}_t \, d\widetilde{M}_{\lambda}$ for every $t \in \R^n$. More precisely, for every $w=(w_1,...,w_k)^t \in \R^{k \cdot 2d}$ with $w_j=(u_j,v_j)^t$ (for $j=1,...k$) we can rewrite \eqref{eq:16072001} as
\begin{equation*} 
\int_{\R^n} \tilde{\psi}_{\lambda} \left(  \sum_{j=1}^{k} \tilde{g}_t(s)^{*} w_j \right) \, ds.
\end{equation*}
In any case, $(Y(t_1),...,Y(t_k))^t$ inherits the $T \alpha S$-property from $(Z(t_1),...,Z(t_k))^t$ due to \cref{15052001}. For the proof of part (b) we fix $t \in \R^n$ and have to show that $Y(t+h)-Y(t)$ converges in distribution to zero as $h \rightarrow 0$. For this purpose define 
\begin{equation*}
A(t,s):= \begin{pmatrix} \cos \skp{t}{s} I_d & \sin \skp{t}{s} I_d \\ - \sin \skp{t}{s} I_d & \cos \skp{t}{s} I_d \end{pmatrix}  \in L(\R^{2d})
\end{equation*}
and recall from Proposition 4.3 in \cite{integral} that the integral in \eqref{eq:15072075} is also linear (as a function of $g$). Then, according to \eqref{eq:16072001} and similarly as in the proof of Proposition 5.2 in \cite{paper2}, it can be computed that the LCF of $Y(t+h)-Y(t)$ is given by $\R^d \ni u \mapsto \int \tilde{\psi}_{\lambda} (\zeta(t,h,s,u) ) \, ds $,
where
\begin{align}
\zeta(t,h,s,u): &=  A(t,s)^{*} \left( (\cos \skp{h}{s} -1) \varphi(s)^{-D^{*}- \frac{q}{\alpha} I}u, -  \sin \skp{h}{s}  \varphi(s)^{-D^{*}- \frac{q}{\alpha} I}u \right)^t  \notag\\
&= ( \tilde{g}_h(s)  A(t,s) )^{*} (u,v)^t  \label{eq:22072073}
\end{align}
and the last-mentioned identity holds true for every $v \in \R^{d}$. For fixed $u \in \R^d$ and due to L\'{e}vy's continuity theorem it remans to show the convergence
\begin{equation} \label{eq:20072011}
 \int \tilde{\psi}_{\lambda} (\zeta(t,h,s,u) ) \, ds \rightarrow 0  \quad (h \rightarrow 0).
 \end{equation}
Observe for every $s \in \R^n$ that we have $\tilde{\psi}_{\lambda} (\zeta(t,h,s,u) )\rightarrow 0$ by continuity of $\tilde{\psi}_{\lambda}$. At the same time it is obvious that \cref{21072080} remains true accordingly for $\tilde{\psi}_{\lambda}$. Hence, if we combine \cref{21072080} with \eqref{eq:22072073} (say for $v=0$), it follows that
\begin{align*}
| \tilde{\psi}_{\lambda} (\zeta(t,h,s,u) | & \le T (1+ \norm{u}^2) (  \norm{\tilde{g}_h(s) A(t,s)}^{\alpha} \wedge \norm{\tilde{g}_h(s) A(t,s)}^{2} ) \\
& \le T (1+ \norm{u}^2) (  \norm{\tilde{g}_h(s)}^{\alpha} \wedge \norm{\tilde{g}_h(s) }^{2} ).
\end{align*}
Note that, in the second step, we used the sub-multiplicativity of the operator norm together with the face that $\norm{A(t,s)}=1$. Then the dominated convergence theorem gives \eqref{eq:20072011} due to \eqref{eq:2207201000}. Finally, the argument for part (c) is mostly presented in the proof of Proposition 5.2 (c) in \cite{paper2} and therefore left to the reader.
 \end{proof}
\begin{remark} \label{21072001}
Besides the fact that we demand $\varphi$ to be $E^{*}$-homogeneous, we can drop condition \eqref{eq:15072065} compared to \cref{15072001}. In addition, we get without further assumptions that $Y(t)$ is full for every $t \ne 0$. Let us also emphasize that we did not require that $D$ belongs to $Q(\R^d)$, which will be explained in \cref{chapter4} more into detail.
\end{remark}
The following observation indicates another difference between both representations. 
\begin{cor} \label{16072088}
Let the assumptions of \cref{15012005} be fulfilled. Additionally, assume that the measure $\tilde{\sigma}$ is invariant under all operators $A \in \mathcal{T}(2d)$, i.e. $A(\tilde{\sigma})=\tilde{\sigma}$, where
\begin{equation*}
\mathcal{T}(2d):= \left \{ \begin{pmatrix} (\cos \beta) I_d & (\sin \beta) I_d \\ - (\sin \beta) I_d & (\cos \beta) I_d \end{pmatrix} : \beta \in [0,2 \pi)  \right \}.
\end{equation*}
Then $\mathbb{Y}$ has stationary increments. 
\end{cor}
\begin{proof}[\bf{Proof of \cref{16072088}}]  \label{proof5}
Fix $t,h \in \R^n$ and observe that, by swapping the roles of $t$ and $h$ in the proof of \cref{15012005}, the LCF of $Y(t+h)-Y(h)$ is given by $ \R^d \ni u \mapsto \int \tilde{\psi}_{\lambda} (\zeta(h,t,s,u) ) \, ds $. Moreover, since $q(r, \theta)=q_{\lambda}(r)=e^{- \lambda r}$ does not depend on $\theta$, it is easy to compute that the given assumption implies that 
\begin{equation*}
\forall A \in \mathcal{T}(2d) \, \, \forall w \in \R^{2d}: \quad \tilde{\psi}_{\lambda}(A^{*}w)=\tilde{\psi}_{\lambda}(w).
\end{equation*}
Hence, since $A(0,s)=I_{2d}$, we observe for every $u \in \R^d$ that $\int \tilde{\psi}_{\lambda} (\zeta(h,t,s,u) ) \, ds$ equals $\int \tilde{\psi}_{\lambda} (\zeta(0,t,s,u) ) \, ds $, which is nothing else than the LCF of $Y(t)$. Finally, in view of \eqref{eq:16072001}, the fdd-extension is obvious and therefore left to the reader. This shows that $\mathbb{Y}$ has stationary increments. 
\end{proof}
\subsection{The special cases $d=1$ and $n=1$}
Certainly, our multivariate approach is very general and its notation therefore rather involved. For many applications it is often enough to consider the cases $d=1$ or $n=1$ (or even $d=n=1$). For those readers we briefly want to remark some useful observations and examples:
\begin{itemize}
\item Recall that we denote the greatest real part of the eigenvalues of $D$ by $H$, which, in case $d=1$, is just a real number. Then $H=h$ is referred to as the so-called \textit{Hurst-index} in literature. Sometimes this Hurst-Index can help to determine possible scaling properties (see \cref{chapter4} below) as well as sample path properties of the corresponding random fields.
\item Consider the situation from \cref{17052077} again, particularly for $d=1$ and $\alpha \ne 1$. In this case we have $\sigma=c (\eps_{-1}+\eps_{1})$ for some $c>0$, where $\eps_x$ denotes the point measure in $x$. Then it follows from Remark 2.8 in \cite{tappe} that the $\psi_{\lambda}$ can be computed as
\begin{equation*}
\psi(u)=- 2c \,  \Gamma(- \alpha) [ \lambda^{\alpha} - \text{Re }  (\lambda + \im u)^{\alpha} ]   , \quad u \in \R.
\end{equation*}
\item On the other hand, in case $n=1$, it appears natural to consider the function $\varphi(\cdot)=|\cdot|$, which is $1$-homogeneous and $(1,1)$-admissible. Then for the existence of the moving-average representation $\mathbb{X}=\{X(t): t \in \R\}$ from \cref{15072001} it only remains to ensure that $0<H<\frac{1}{\alpha}+ \frac{1}{2}$ holds true. Moreover, $X(t)$ is full for every $t \ne 0$ provided that $H \ne \frac{1}{\alpha}$. Conversely, in the context of \cref{15012005}, condition \eqref{eq:14072099} is equivalent to $\frac{1}{2}-\frac{1}{\alpha}<h \le H<1$ in this case.
\end{itemize}
%% Section 4: 
\section{The role of the tempering parameter} \label{chapter4}
Recall the definitions from \cref{chapter3}, particularly those of $\mathbb{M}_{\lambda}$, $\widetilde{\mathbb{M}}_{\lambda},\mathbb{M}_{\lambda}^{\C}$ and the respective generators. In either case we call $\lambda>0$ the \textit{tempering parameter}. This section will essentially deal with the behavior of the random fields from \cref{chapter3} as $\lambda \rightarrow 0$. For this purpose we first have to define further generators and corresponding random measures that deal with the case $\lambda=0$ in a reasonable way.
\begin{remark} \label{17072001}
Roughly speaking, for $q \equiv 1$ in \eqref{eq:17052010} we get back the class of $\alpha$-stable distributions on $\R^d$, although this $q$ is no valid tempering function anymore. Anyway, let $\mu_0 \sim [0,0,\phi_0]$, where $\phi_0(dr,d \theta):=r^{-\alpha-1} \, dr \, \sigma(d\theta)$ and note that $\mu_0$ is symmetric again. Similarly to \cref{15052002} there exists an $\R^d$-valued random measure, which is generated by $\mu_0$ and which we denote by $\mathbb{M}_0$. We call $\mathbb{M}_0$ a symmetric $\alpha$-stable (abbreviated by $S \alpha S$) ISRM. Note that $\mathbb{M}_0$ corresponds to Example 2.4 in \cite{paper2}. There it has been shown that a measurable function $f:S \rightarrow L(\R^d)$ is \textit{integrable} with respect to $\mathbb{M}_0$ (i.e. $f \in \mathcal{I}(\mathbb{M}_0)$) if and only if
\begin{equation*}
 \int_S \int_{S^{d-1}} \norm{f(s) \theta}^{\alpha} \, \sigma(d \theta) \, \nu(ds)=  \int_{\R^n} \int_{S^{d-1}} \norm{f(s) \theta}^{\alpha} \, \sigma(d \theta) \, ds <\infty. 
\end{equation*}
In view of \cref{15052076} and \cref{17052077} it easily follows that $\mathcal{I}(\mathbb{M}_0) \subset \mathcal{I}(\mathbb{M}_{\lambda})=\mathcal{I}(\mathbb{M}_{1})$ for every $\lambda>0$. Finally, in a similar way as before we obtain an $\R^{2d}$-valued $S \alpha S$ random measure $\widetilde{\mathbb{M}}_0$ (generated by $\tilde{\mu}_0$ with LCF $\tilde{\psi}_0$) and a corresponding $\C^{d}$-valued one, denoted by $\mathbb{M}^{\C}_0$. We omit the details. 
\end{remark}

\subsection{A note on different kinds of tempering} 
Let us emphasize that a main advantage of our approach is based on the fact that we use $T \alpha S$ distributions as a generator for the underlying random measures. By doing so we combine both the heavy tailed behavior of classical $\alpha$-stable distributions and Gaussian trends. Hence, concerning the corresponding stochastic integrals, the observation of \cref{17072001} above is not surprising. Namely we enlarge the class of possible kernel functions when using $T \alpha S$ generators instead of $\alpha$-stable ones. Let us illustrate this effect by means of the kernel functions from \cref{chapter3}. \\
For this purpose note that the following statement is due to Theorem 2.5 and Theorem 2.6 in \cite{lixiao}, respectively, except for a negligible difference. That is, the authors in \cite{lixiao} assume $\sigma$ (and $\tilde{\sigma}$) to be uniformly distributed on $S^{d-1}$ (and $S^{2d-1}$), at least up to a constant, i.e. the LCF of $\mu_0$ is given by $\psi_0(u)=- \rho \norm{u}^{\alpha}$ for some $\rho>0$. 
\begin{prop} \label{230720111}
\begin{itemize}
\item[(a)] Assume that the conditions of \cref{15072001} are fulfilled with $H < \beta$ instead of \eqref{eq:14072005}. Then the random field $\mathbb{X}_0=\{X_0(t): t \in \R^n\}$ exists, where
\begin{equation*}
X_0(t):= \int_{\R^n} \left[  \varphi(t-s)^{D- \frac{q}{\alpha}  I} - \varphi(-s)^{D-  \frac{q}{\alpha} I}   \right] \, \mathbb{M}_0(ds) , \quad t \in \R^n.
\end{equation*}
\item[(b)] Assume that the conditions of \cref{15012005} are fulfilled for some $D \in Q(\R^d)$ (i.e. \eqref{eq:14072099} reduces to $H<a$). Then the random field $\mathbb{Y}_0=\{Y_0(t): t \in \R^n\}$ exists, where
\begin{equation*}
Y_0(t):= \text{Re } \int_{\R^n} \left(e^{ \im \skp{t}{s}} -1 \right) \varphi(s)^{-D- \frac{q}{\alpha} I}  \, \mathbb{M}_0^{\C}(ds) , \quad t \in \R^n.
\end{equation*}
\end{itemize}
In both cases, $\mathbb{X}_0$ and $\mathbb{Y}_0$ are $S \alpha S$ random fields, respectively. 
\end{prop}
Here the notion of \textit{$S \alpha S$ random fields} is analogous to \cref{14072001}. Moreover, as expected, the assumptions under which $\mathbb{X}_0$ and $\mathbb{Y}_0$ exist are more restrictive compared to their natural counterparts from \cref{chapter3} and we lose the $T \alpha S$ property. However, the remaining properties from \cref{15072001} and \cref{15012005} still apply to $\mathbb{X}_0$ and $\mathbb{Y}_0$ accordingly.\\
Actually, the random fields from \cref{chapter3} have a \textit{true} tempered stable character that arises from \cref{15052001}. Let us be more specific what we mean by the word \textit{true}: Recall the main results from \cite{didier} and note that there a different kind of tempering is considered, which extends the univariate idea from \cite{mark}. That is, the authors in \cite{didier} somehow add a tempering function to the actual kernel function, say $f(\cdot)$, while they abstain from the use of $T \alpha S$ random measured as integrator. This has two consequences. On the one hand, it follows that the behavior of $f(s)$ is restrained for large values of $s$. However, on the other hand, the resulting stochastic integrals have no $T \alpha S$ distribution. \\
Also note that, in \cite{didier}, the aforementioned approach has been performed both for a moving-average representation and for a harmonizable representation. We focus on the first one. Then, taking into account the different notation in \cite{didier}, it is easy to verify that we obtain the following observation as a by-product of Theorem 3.1 in \cite{didier} (namely for $B= \frac{1}{\alpha} I$). The details are left to the reader.
\begin{example} \label{24072001}
Fix $\lambda>0$ and let $\varphi: \R^n \rightarrow [0,\infty)$ be an $E$-homogeneous function for some $E \in Q(\R^n)$, where $q$ denotes the trace of $E$ again. Also consider $D \in Q(\R^d)$ arbitrary. Then the random field $\hat{\mathbb{X}}=\{\hat{X}(t): t \in \R^n\}$, defined by
\begin{equation*}
\hat{X}(t):= \int_{\R^n} \left[  e^{- \lambda \varphi(t-s)}  \varphi(t-s)^{D- \frac{q}{\alpha}  I} - e^{- \lambda \varphi(-s)} \varphi(-s)^{D-  \frac{q}{\alpha} I}   \right] \, \mathbb{M}_0(ds), 
\end{equation*}
exists, where $\mathbb{M}_0$ is as above. Moreover, $\hat{\mathbb{X}}$ is a $S \alpha S$ random field.
\end{example}
Recall from the proof of Theorem 2.5 in \cite{lixiao} that, in the context of \cref{230720111} (a), condition $H < \beta$ is needed in order to control the behavior of $f_t(s)$ from \eqref{eq:15072001} for large values of $s$. Hence, as already announced before, this condition disappears when tempering in the sense of \cite{didier}. For the same reason we observe in \cref{24072001} that $\varphi$ has not to be admissible anymore.
\subsection{Scaling properties and tangent fields}
In what follows we will show that the random fields from \cref{chapter3} have an intimate relation to those from \cref{230720111}. 
\begin{lemma}  \label{16072044}
We have that 
\begin{equation*}
\forall \rho, \lambda>0 \,  \, \forall u \in \R^d: \quad \psi_{\lambda}(\rho \, u)= \rho^{\alpha} \, \psi_{\lambda/ \rho} (u).
\end{equation*}
Similarly, the statement holds true for $\tilde{\psi}_{\lambda}$.
\end{lemma}
\begin{proof}
In view of $\mu_{\lambda} \sim [0,0, \phi_{\lambda}]$, where $\phi_{\lambda}$ is symmetric, we obtain that
\begin{equation*}
\psi_{\lambda} (\rho u) = \int_{\R^d} (\cos \skp{\rho u}{x} -1) \, \phi_{\lambda}(dx)  = \int_{\R^d} (\cos \skp{ u}{x}-1)  \,(\rho \phi_{\lambda}) (dx).
\end{equation*}
It remains to show that $(\rho \phi_{\lambda})= \rho^{\alpha} \cdot \phi_{\lambda/ \rho}$. Fix $A \in \B(\R^d)$. Then, due to \eqref{eq:17052010} and by a change of variables, we compute that
\begin{align*}
(\rho \phi_{\lambda})(A) &= \int_{S^{d-1}} \int_{0}^{\infty} \indikator{\rho^{-1} A}{r \theta} r^{- \alpha -1} e^{- \lambda r} \,dr \, \sigma (d \theta) \\
& = \int_{S^{d-1}} \int_{0}^{\infty} \indikator{ A}{\rho r \theta} r^{- \alpha -1} e^{- \lambda r} \,dr \, \sigma (d \theta) \\
& = \rho^{\alpha} \int_{S^{d-1}} \int_{0}^{\infty} \indikator{ A}{ r \theta} r^{- \alpha -1} e^{- \lambda r/ \rho} \,dr \, \sigma (d \theta) = \rho^{\alpha} \cdot \phi_{\lambda/ \rho} (A).
\end{align*}
\end{proof}
In order to be more specific we denote by $\mathbb{X}_{\lambda}=\{ X_{\lambda}(t): t \in \R^n \}$ and $\mathbb{Y}_{\lambda}=\{ Y_{\lambda}(t): t \in \R^n \}$ the moving-average and harmonizable representation corresponding to $\mathbb{M}_{\lambda}$ and $\mathbb{M}_{\lambda}^{\C}$ attained in \cref{chapter3}, respectively. Then the next result states scaling properties for $\mathbb{X}_{\lambda}$ and $\mathbb{Y}_{\lambda}$ and is therefore of independent interest. 
\begin{prop} \label{17072077}
\begin{itemize}
\item[(a)] Under the assumptions of \cref{15072001} we have that
\begin{equation}  \label{eq:224072077}
\forall c>0: \quad \{X_{\lambda} (c^E t): t \in \R^n \}  \overset{fdd}{=}  \{c^D X_{c^{ \frac{q}{\alpha}} \lambda} ( t): t \in \R^n \}.   
\end{equation}
\item[(b)] Conversely, under the assumptions of \cref{15012005} we have that
 \begin{equation}  \label{eq:224072078}
 \forall c>0: \quad \{Y_{\lambda} (c^E t): t \in \R^n \}  \overset{fdd}{=}  \{c^D Y_{c^{ - \frac{q}{\alpha}} \lambda} ( t): t \in \R^n \}.  
  \end{equation}
\end{itemize}
\end{prop}
\begin{proof}
For the proof of part (a) recall $f_t(\cdot)$ from \eqref{eq:15072001} and observe for (almost) all $s \in \R^n$ that, by $E$-homogeneity of $\varphi$, we have
\begin{equation} \label{eq:16072042}
f_{c^E t} (s)=c^{D- \frac{q}{\alpha} I} f_t(c^{-E} s).
\end{equation}
Now fix $k \in \N$ and $t_1,...,t_k \in \R^n$. Then, due to \eqref{eq:17052001}, \eqref{eq:16072042}, \cref{16072044} and by a change of variables, we compute that the LCF of $(X_{\lambda}(c^E t_1),...,X_{\lambda}(c^E t_k))^t$ is given by 
\begin{align}
\R^{k \cdot d} \ni (u_1,...,u_k) \mapsto & \int_{\R^n} \psi_{\lambda} \left( \sum_{j=1}^{k} f_{c^E t_j}(s)^{*}u_j \right) \, ds \label{eq:03082001}  \\
&= \int_{\R^n} \psi_{\lambda} \left(c^{- \frac{q}{\alpha}} \sum_{j=1}^{k} f_{ t_j}(c^{-E} s)^{*} c^{D^{*}} u_j \right) \, ds  \notag \\
&=c^{-q} \int_{\R^n} \psi_{c^{ \frac{q}{\alpha}} \lambda } \left(\sum_{j=1}^{k} f_{ t_j}(c^{-E} s)^{*} c^{D^{*}} u_j \right) \, ds   \notag \\
&= \int_{\R^n} \psi_{c^{ \frac{q}{\alpha}} \lambda} \left(\sum_{j=1}^{k} f_{ t_j}( s)^{*} c^{D^{*}} u_j \right) \, ds , \notag
\end{align}
which, again by \eqref{eq:17052001}, is just the LCF of $(c^D X_{c^{ \frac{q}{\alpha}} \lambda}( t_1),...,c^D X_{c^{ \frac{q}{\alpha}} \lambda}( t_k))^t$. This gives \eqref{eq:224072077}. Concerning part (b) recall the representation of $Y(t)=Y_{\lambda}(t)$ from \eqref{eq:17072001}. Using the $E^{*}$-homogeneity of $\varphi$ (which, in general, is different from the function $\varphi$ in part (a)) we obtain that
\begin{equation}  \label{eq:240720200}
(\cos \skp{c^E t}{s}-1 ) \phi(s)^{-D^{*}- \frac{q}{\alpha} I}=c^{\frac{q}{\alpha}} (\cos \skp{  t}{c^{E^{*}} s}-1 )   \phi(c^{E^{*}}s)^{-D^{*}  - \frac{q}{\alpha} I} c^{D^{*}}
\end{equation}
and, similarly, that 
\begin{equation} \label{eq:240720201}
\sin \skp{c^E t}{s}  \varphi(s)^{-D^{*}-\frac{q}{\alpha} I }= c^{\frac{q}{\alpha}} \sin \skp {t}{c^{E^{*}}s}  \varphi(c^{E^{*}}s)^{-D^{*}  -\frac{q}{\alpha} I } c^{D^{*}}  .
\end{equation}
Note that the trace of $E^{*}$ equals $q$ again. Hence, based on \eqref{eq:16072001} and \eqref{eq:240720200}-\eqref{eq:240720201}, we can proceed as in the proof of part (a) in order to verify the accuracy of \eqref{eq:224072078}. The details are left to the reader.
\end{proof}
\begin{remark}
Observe that there is a striking difference in the scaling behavior between the moving-average and the harmonizable representation of $T \alpha S$ random fields considered in \cref{chapter3}.
\end{remark}
We need another auxiliary result, which is of independent interest again. For this purpose recall \cref{230720111} and that the conditions mentioned there (and quoted below)  always imply the existence of the associated $T \alpha S$ representations from \cref{chapter3}. Here, $\overset{fdd}{\Rightarrow}$ means convergence of all finite-dimensional distributions. 
\begin{lemma} \label{17072067}
\begin{itemize}
\item[(a)] Assume that the conditions of \cref{15072001} are fulfilled with $H < \beta$ instead of \eqref{eq:14072005}. Then, as $\lambda \rightarrow 0$, we have that $\mathbb{X}_{\lambda} \overset{fdd}{\Rightarrow} \mathbb{X}_0$.
\item[(b)] Assume that the conditions of \cref{15012005} are fulfilled for some $D \in Q(\R^d)$. Then, as $\lambda \rightarrow 0$, we have that $\mathbb{Y}_{\lambda} \overset{fdd}{\Rightarrow} \mathbb{Y}_0$.
\end{itemize}
\end{lemma}
\begin{proof} We only prove part (a) since the proof of part (b) is very similar and, moreover, mostly covered by the ideas in the proof of \cref{03082077} below. Fix $k \in \N$ as well as $t_1,...,t_k \in \R^n$ and $u_1,...,u_k \in \R^d$. Then, in view of \eqref{eq:03082001} (for $c=1$), the assertion would follow from
\begin{equation} \label{eq:03082002}  
 \int_{\R^n} \psi_{\lambda} \left( \sum_{j=1}^{k} f_{t_j}(s)^{*}u_j \right) \, ds \rightarrow  \int_{\R^n} \psi_{0} \left( \sum_{j=1}^{k} f_{t_j}(s)^{*}u_j \right) \, ds \quad (\lambda \rightarrow 0)
 \end{equation}
due to L\'{e}vy's continuity theorem. Using \eqref{eq:17052010} and \eqref{eq:14052005} we first observe for every $\lambda>0$ and $u \in \R^d$ that
\begin{align}
|\psi_{\lambda}(u)|=- \psi_{\lambda}(u) & = \int_{S^{d-1}} \int_{0}^{\infty} (1- \cos \skp{u}{r \theta} ) r^{- \alpha -1} e^{- \lambda r} \, dr \, \sigma (d \theta) \notag \\
&  \le   \int_{S^{d-1}} \int_{0}^{\infty} (1- \cos \skp{u}{r \theta} ) r^{- \alpha -1} \, dr \, \sigma (d \theta)= - \psi_{0}(u)  = |\psi_{0}(u)|. \label{eq:03082099}
\end{align}
In particular, the dominated convergence theorem implies for every $u \in \R^d$ that we have $\psi_{\lambda}(u) \rightarrow \psi_0(u)$ as $\lambda \rightarrow 0$. At the same
time it is well-known that we have $|\psi_0(u)| \le \rho \norm{u}^{\alpha}$, where $\rho>0$ is a constant (see the proof of Example 2.4 in \cite{paper2}), and that the inequality $1-\cos (a+b) \le 2(2- \cos (a)- \cos(b))$ holds true (see Proposition 1.3.4 in \cite{thebook}). Combine this with \eqref{eq:03082099} to verify the accuracy of
\begin{equation} \label{eq:03082085}
\left|  \psi_{\lambda} \left( \sum_{j=1}^{k} f_{t_j}(s)^{*}u_j \right) \right| \le 2^{k-1} \sum_{j=1}^{k} \left |\psi_0 (f_{t_j}(s)^{*}u_j)  \right| \le 2^{k-1} \rho \sum_{j=1}^{k} \norm{u_j}^{\alpha} \norm{f_{t_j}(s)}^{\alpha}.
\end{equation}
In addition and under the present assumptions, Theorem 2.5 in \cite{lixiao} states that the function on the right-hand side of \eqref{eq:03082085} belongs to $\mathcal{L}^1(ds)$. Using the dominated convergence theorem again this gives \eqref{eq:03082002}.
\end{proof}
The foregoing observations can be used to go a step further, leading to the notion of so-called \textit{tangent fields}. Note that, in the univariate case, the idea goes back to \cite{falconer4}. The following definition turns out to be a multivariate extension of it and is essentially due to \cite{multi}.
\begin{defi} \label{24072066}
Fix $x \in \R^n$ and let $\mathbb{X}=\{X(t):t \in \R^n\}$ as well as $\mathbb{X}'=\{X'(t):t \in \R^n\}$ be $\R^d$-valued random fields. Then, for $E \in Q(\R^n)$ and $D \in Q(\R^d)$, we say that $\mathbb{X}$ is \textit{$(E,D)$-localisable at x} with \textit{local form}/\textit{tangent field} $\mathbb{X}'$ if we have that
\begin{equation} \label{eq:240720987} 
\{c^{-D}(X(x+c^{E}t)-X(x)): t \in \R^n \} \overset{fdd}{\Rightarrow} \{X'(t):t \in \R^n\} \quad (c \rightarrow 0).
\end{equation}
\end{defi}
Let us emphasize that, in general, the tangent field $\mathbb{X}'$ as well as the corresponding operators $E$ and $D$ in \eqref{eq:240720987} do depend on the point $x$ at which we have localisability. However, this is not the case in our setting as the following main result demonstrates.
\begin{theorem}
Assume that the conditions of \cref{15072001} are fulfilled with $H < \beta$ instead of \eqref{eq:14072005}. Then we have for every $\lambda>0$ and $x \in \R^n$ that the moving-average representation $\mathbb{X}_{\lambda}$ is $(E,D)$-localisable at $x$ with tangent field $\mathbb{X}'=\mathbb{X}_0$ (not depending on $x$).
\end{theorem}
\begin{proof}
Fix $k \in \N$ and consider $t_1,...,t_k \in \R^n$. Then, using the fact that $\mathbb{X}_{\lambda}$ has stationary increments together with \eqref{eq:224072077}, it is easy to check that we have
\begin{equation*}
\forall c>0: \quad \{c^{-D}(X_{\lambda}(x+c^{E}t)-X_{\lambda}(x)): t \in \R^n \} \overset{fdd}{=}  \{X_{c^{\frac{q}{\alpha}} \lambda}(t): t \in \R^n \}.
\end{equation*}
Thus the assertion follows from \cref{17072067}. 
\end{proof}
Recall \eqref{eq:240720987} and observe for all $x,t \in \R^n$ that $x+c^E t$ contracts to $x$ as $c \rightarrow 0$, since $E \in Q(\R^n)$ and due to Lemma 4.2 in \cite{lixiao}. Hence, it is crucial to emphasize that the convergence in \eqref{eq:24072080} below is meant as $c \rightarrow \infty$, which is striking differently.  Accordingly, we do no longer use the term \textit{tangent field} hereinafter.
\begin{theorem} \label{03082077}
Assume that the conditions of \cref{15012005} are fulfilled for some $D \in Q(\R^d)$. Then we have for every $\lambda>0$ and $x \in \R^n$ that
\begin{equation}  \label{eq:24072080}
\{c^{-D}(Y_{\lambda}(x+c^{E}t)-Y_{\lambda}(x)): t \in \R^n \} \overset{fdd}{\Rightarrow} \mathbb{Y}_0 \quad (c \rightarrow \infty).
\end{equation}
\end{theorem}
\begin{proof}
Let us recall the proof of \cref{15012005}, particularly \eqref{eq:16072001} and \eqref{eq:22072073}. Then, for $t_1,...,t_k \in \R^n$ as before, we compute that the LCF of $(c^{-D}(Y_{\lambda}(x+c^{E}t_j)-Y_{\lambda}(x)) : 1 \le j \le k)^t$ is given by
\begin{equation*}
\R^{k \cdot d} \ni (u_1,...,u_k)^t \mapsto \int_{\R^n} \tilde{\psi}_{\lambda} \left (\sum_{j=1}^k  \zeta(x,c^E t_j,s,c^{-D^{*}}u_j) \right) \, ds.
\end{equation*}
At the same time, by definition of $A(x,s)$ and $\zeta$, we can use \eqref{eq:240720200}-\eqref{eq:240720201} together with \cref{16072044} to obtain that
\begin{align*}
 \int_{\R^n} \tilde{\psi} \left (\sum_{j=1}^k  \zeta(x,c^E t_j,s,c^{-D^{*}}u_j) \right) \, ds &=  \int_{\R^n} \tilde{\psi}_{\lambda} \left (c^{\frac{q}{\alpha}}  \sum_{j=1}^k  \zeta(c^{-E}x, t_j,c^{E^{*}}s, u_j) \right) \, ds \\
 &=  c^{q} \int_{\R^n} \tilde{\psi}_{c^{-q/ \alpha} \lambda} \left ( \sum_{j=1}^k  \zeta(c^{-E}x, t_j,c^{E^{*}}s, u_j) \right) \, ds  \\
  &=   \int_{\R^n} \tilde{\psi}_{c^{-q/ \alpha} \lambda} \left ( \sum_{j=1}^k  \zeta(c^{-E}x, t_j,s, u_j) \right) \, ds ,
\end{align*}
where in the last step we performed a change of variables. Fix $u_1,...,u_k \in \R^d$. Then, by L\'{e}vy's continuity theorem and in view of $A(0,s)=I_{2d}$, it obviously remains to establish that
\begin{equation} \label{eq:04082092}
\int_{\R^n} \tilde{\psi}_{c^{-q/ \alpha} \lambda} \left ( \sum_{j=1}^k  \zeta(c^{-E}x, t_j,s, u_j) \right) \, ds \rightarrow \int_{\R^n} \tilde{\psi}_{0} \left ( \sum_{j=1}^k  \zeta(0, t_j,s, u_j) \right) \, ds \quad (c \rightarrow \infty).
\end{equation}
First observe for every $w \in \R^{2d}$ and as $\lambda \rightarrow 0$ that, as in the proof of \cref{17072067}, we have $\tilde{\psi}_{\lambda}(w) \rightarrow \tilde{\psi}_0(w)$. That is, $\tilde{\mu}_{\lambda}$ converges to $\tilde{\mu}_{0}$ weakly due to L\'{e}vy's continuity theorem again. Then Lemma 3.1.10 in \cite{thebook} states that the convergence $\tilde{\psi}_{\lambda} \rightarrow \tilde{\psi}_0$ holds even uniformly on compact subsets of $\R^{2d}$. Hence, since $c^{-E} x \rightarrow 0$ as $c \rightarrow \infty$ (see Lemma 4.2 in \cite{lixiao}),  it is easy to check for every $s \in \R^n$ that we have
\begin{equation*}
 \tilde{\psi}_{c^{-q/ \alpha} \lambda} \left ( \sum_{j=1}^k  \zeta(c^{-E}x, t_j,s, u_j) \right) \rightarrow  \tilde{\psi}_{0} \left ( \sum_{j=1}^k  \zeta(0, t_j,s, u_j) \right) \quad (c \rightarrow \infty).
\end{equation*}
Use essentially the same argument as in \eqref{eq:03082085} to verify that there exists a constant $L>0$ (only depending on $k$) such that, for every $c>0$ and $s \in \R^n$, we have
\begin{align}
\left| \tilde{\psi}_{c^{-q/ \alpha} \lambda} \left ( \sum_{j=1}^k  \zeta(c^{-E}x, t_j,s, u_j) \right)  \right| & \le L \sum_{j=1}^k \norm{\zeta(c^{-E}x, t_j,s, u_j)}^{\alpha} \notag \\
& \le L \sum_{j=1}^k  \norm{A(c^{-E},x)}^{\alpha} \norm{\tilde{g}_{t_j}(s)}^{\alpha} \norm{u_j}^{\alpha}  \notag \\
&=  L \sum_{j=1}^k  \norm{\tilde{g}_{t_j}(s)}^{\alpha} \norm{u_j}^{\alpha}  , \label{eq:040820008}
\end{align}
where in the second step we benefited from \eqref{eq:22072073} together with the sup-multiplicativity of the operator norm. Also note that we used $\norm{A(c^{-E},x)}^{\alpha}=1$ again. Anyway and based on \eqref{eq:04082001}, it was just the outcome of Theorem 2.6 in \cite{lixiao} that the function in \eqref{eq:040820008} belongs to $\mathcal{L}^1(ds)$. Eventually, the dominated convergence theorem gives \eqref{eq:04082092}.
\end{proof}
\begin{remark}
Observe that the proof of \cref{03082077} is simpler if we additionally assume that $\mathbb{Y}$ has stationary increments as in \cref{16072088}. In fact, by \eqref{eq:224072078} and \cref{17072067} we get in this particular case that
\begin{align*}
\{c^{-D}(Y_{\lambda}(x+c^{E}t)-Y_{\lambda}(x)): t \in \R^n \} & \overset{fdd}{=}  \{c^{-D} Y_{ \lambda}(c^{E}t): t \in \R^n \}   \\
& \overset{fdd}{=}  \{Y_{c^{- \frac{q}{\alpha}} \lambda}(t): t \in \R^n \} \overset{fdd}{\Rightarrow} \{Y_0(t):t \in \R^n \}
\end{align*}
as $c \rightarrow \infty$.
\end{remark}
\secret{
\section*{Acknowledgement}
To do
}
%% References


\begin{thebibliography}{10}

\bibitem{mark1}
M.~S. Alrawashdeh, J.~F. Kelly, M.~M. Meerschaert, and H.-P. Scheffler.
\newblock Applications of inverse tempered stable subordinators.
\newblock {\em Computers \& Mathematics with Applications}, 73(6):892--905,
  2017.

\bibitem{mark5}
B.~Baeumer and M.~M. Meerschaert.
\newblock Tempered stable l{\'e}vy motion and transient super-diffusion.
\newblock {\em Journal of Computational and Applied Mathematics},
  233(10):2438--2448, 2010.

\bibitem{BiLa09}
H.~Bierm{\'e} and C.~Lacaux.
\newblock H{\"o}lder regularity for operator scaling stable random fields.
\newblock {\em Stochastic Processes and their Applications}, 119(7):2222--2248,
  2009.

\bibitem{bms}
H.~Bierm{\'e}, M.~M. Meerschaert, and H.-P. Scheffler.
\newblock Operator scaling stable random fields.
\newblock {\em Stochastic Processes and their Applications}, 117(3):312--332,
  2007.

\bibitem{mark4}
A.~Chakrabarty and M.~M. Meerschaert.
\newblock Tempered stable laws as random walk limits.
\newblock {\em Statistics \& probability letters}, 81(8):989--997, 2011.

\bibitem{didier}
G.~Didier, S.~Kanamori, and F.~Sabzikar.
\newblock On multivariate fractional random fields: tempering and
  operator-stable laws.
\newblock {\em arXiv preprint arXiv:2002.09612}, 2020.

\bibitem{Dud02}
R.~M. Dudley.
\newblock {\em Real analysis and probability}, volume~74.
\newblock Cambridge University Press, 2002.

\bibitem{falconer4}
K.~J. Falconer.
\newblock Tangent fields and the local structure of random fields.
\newblock {\em Journal of Theoretical Probability}, 15(3):731--750, 2002.

\bibitem{jameson}
G.~Jameson.
\newblock The incomplete gamma functions.
\newblock {\em The Mathematical Gazette}, 100(548):298--306, 2016.

\bibitem{multi}
D.~Kremer and H.-P. Scheffler.
\newblock Multi operator-stable random measures and fields.
\newblock {\em Stochastic Models}, 35(4):429--468, 2019.

\bibitem{integral}
D.~Kremer and H.-P. Scheffler.
\newblock Multivariate stochastic integrals with respect to independently
  scattered random measures on $\delta$-rings.
\newblock {\em Publicationes Mathematicae Debrecen}, 95(1-2):39--66, 2019.

\bibitem{paper2}
D.~Kremer and H.-P. Scheffler.
\newblock Operator-stable and operator-self-similar random fields.
\newblock {\em Stochastic Processes and their Applications},
  129(10):4082--4107, 2019.

\bibitem{tappe}
U.~K{\"u}chler and S.~Tappe.
\newblock Tempered stable distributions and processes.
\newblock {\em Stochastic Processes and their Applications},
  123(12):4256--4293, 2013.

\bibitem{lixiao}
Y.~Li and Y.~Xiao.
\newblock Multivariate operator-self-similar random fields.
\newblock {\em Stochastic Processes and their Applications}, 121(6):1178--1200,
  2011.

\bibitem{mark2}
M.~M. Meerschaert, P.~Roy, and Q.~Shao.
\newblock Parameter estimation for tempered power law distributions.
\newblock {\em Communications in statistics--theory and methods}, 2009.

\bibitem{mark}
M.~M. Meerschaert and F.~Sabzikar.
\newblock Tempered fractional stable motion.
\newblock {\em Journal of Theoretical Probability}, 29(2):681--706, 2016.

\bibitem{thebook}
M.~M. Meerschaert and H.-P. Scheffler.
\newblock {\em Limit distributions for sums of independent random vectors:
  Heavy tails in theory and practice}, volume 321.
\newblock John Wiley \& Sons, 2001.

\bibitem{mark3}
M.~M. Meerschaert and A.~Sikorskii.
\newblock {\em Stochastic models for fractional calculus}, volume~43.
\newblock Walter de Gruyter, 2011.

\bibitem{mark6}
M.~M. Meerschaert, Y.~Zhang, and B.~Baeumer.
\newblock Tempered anomalous diffusion in heterogeneous systems.
\newblock {\em Geophysical Research Letters}, 35(17), 2008.

\bibitem{RajRo89}
B.~S. Rajput and J.~Rosinski.
\newblock Spectral representations of infinitely divisible processes.
\newblock {\em Probability Theory and Related Fields}, 82(3):451--487, 1989.

\bibitem{rosinski}
J.~Rosi{\'n}ski.
\newblock Tempering stable processes.
\newblock {\em Stochastic processes and their applications}, 117(6):677--707,
  2007.

\bibitem{SaTaq94}
G.~Samoradnitsky and M.~S. Taqqu.
\newblock {\em Stable non-Gaussian random processes: stochastic models with
  infinite variance}, volume~1.
\newblock CRC press, 1994.

\bibitem{sato}
K.-I. Sato.
\newblock {\em L{\'e}vy processes and infinitely divisible distributions}.
\newblock Cambridge university press, 1999.

\bibitem{temme}
N.~Temme.
\newblock Uniform asymptotic expansions of the incomplete gamma functions and
  the incomplete beta function.
\newblock {\em Mathematics of Computation}, 29(132):1109--1114, 1975.

\end{thebibliography}
\end{document}